\documentclass[a4paper]{amsart}
\usepackage{tikz}
\usetikzlibrary{quotes}
\usepackage{biblatex}
\usepackage{amsaddr}
\usepackage{mathrsfs} 
\usepackage{array} 
\usepackage{amssymb} 
\usepackage{stmaryrd} 
\usepackage{pdflscape}  
\usepackage{tikz-cd}
\usepackage[left=1.5in,right=1.5in,top=1in,bottom=1in]{geometry}
\usepackage{hyperref}
\usepackage{parskip}
\numberwithin{equation}{section}

\newcommand{\Stirling}[0]{\genfrac\{\}{0pt}{}}
\newcommand{\Stirlingone}[0]{\genfrac[]{0pt}{}}
\newcommand{\Lah}[0]{\genfrac{\lfloor}{\rfloor}{0pt}{}}
\newcommand{\mStirling}[0]{\genfrac{\llbracket}{\rrbracket}{0pt}{}}
\newcommand{\mStirlingone}[2]{\left[\mkern-3.8mu\middle|\genfrac{}{}{0pt}{}{#1}{#2}\middle|\mkern-3.8mu\right]}
\newcommand{\mLah}[2]{\left\lfloor\mkern-5.5mu\left\lfloor\genfrac{}{}{0pt}{}{#1}{#2}\right\rfloor\mkern-5.5mu\right\rfloor}
\newcommand{\uuline}[1]{\underline{\underline{#1}}}
\newcommand{\ooline}[1]{\overline{\overline{#1}}}

\newcommand{\floor}[1]{\left\lfloor #1 \right\rfloor}

\newcommand{\BINOM}{\mathrm{BINOM}}
\newcommand{\Lshift}{\mathrm{L}}

\newtheorem{theorem}{Theorem}[section]
\newtheorem{definition}[theorem]{Definition}
\newtheorem{proposition}[theorem]{Proposition}

\newtheorem{lemma}[theorem]{Lemma}
\newtheorem{corollary}[theorem]{Corollary}
\newtheorem{remark}[theorem]{Remark}
\newtheorem{problem}[theorem]{Problem}

\title[$m$-Bell and $m$-Stirling numbers]{$m$-Bell and $m$-Stirling numbers: iterated binomial transforms, hyper-Bessel functions, and moments of the Conway--Maxwell--Poisson distribution}
\author{Vencislav Popov}
\email{ven.popov@uzh.ch}
\thanks{Computational results and other supporting information are available on \href{https://github.com/venpopov/generalization-bell-numbers}{GitHub}. Claude Fable 5 (Anthropic) assisted with some of the derivations; its contributions are documented in the repository's pull requests, whereas the human author's contributions are direct commits, so the repository history provides a record of the respective contributions.}
\date{}

\begin{document}
\begin{abstract}
We introduce a natural generalization of the Bell numbers: the $m$-Bell numbers $B^{(m)}_{n}$, characterized by the property that $m$ applications of the binomial transform reproduce the original sequence shifted $m$ places to the left. Their exponential generating functions satisfy $m$-th order ordinary differential equations whose solutions are hypergeometric (hyper-Bessel) functions, specializing to the exponential function when $m=1$ (classical Bell numbers) and to modified Bessel functions when $m=2$ (yielding ``Bessel--Bell'' numbers). Mirroring the Bell--Stirling correspondence, we construct $m$-Stirling triangular arrays from the two-term recurrence $\mStirling{n+1}{k}_m=m\floor{k/m}\mStirling{n}{k}_m+\mStirling{n}{k-1}_m$ and prove an elementary shift identity from which the central structure theorem follows: the row sums of the $m$-Stirling triangle reproduce $B^{(m)}_{n}$, and, more finely, the residue-class row sums are precisely the $m$ primitive $m$-Bell sequences. The $m$-Stirling numbers come in dual pairs (with first-kind partners, generalized falling factorials, and Lah-type companions), serve as conversion operators between polynomial bases, admit Dobi\'nski-like formulas, and count congruence-constrained partitions in an urn model as well as restricted permutation insertion histories. Finally, we show that the $m$-Bell numbers govern the moments of the Conway--Maxwell--Poisson distribution with integer dispersion parameter $\nu=m$: the scaled moments are combinations of fixed hyper-Bessel carrier ratios whose integer coefficients are precisely the primitive $m$-Bell sequences, recovering for $m=1$ the classical fact that the moments of the Poisson distribution are the Bell numbers.
\end{abstract}


\maketitle

\section{Introduction}\label{sec-introduction}
For a sequence  $a = (a_n)_{n \geq 0}$, its binomial transform $b = \mathrm{BINOM}a$ is:
\begin{equation*}
    b_n = \sum_{k=0}^{n} \binom{n}{k} a_k.
\end{equation*}
After $m \ge 1$ iterations \cite{spiveyKbinomialTransformsHankel2006}:
\begin{equation*}
    (\mathrm{BINOM}^m a)_n =\sum_{k=0}^{n} \binom{n}{k} m^{n-k} a_k,
\end{equation*}
Bernstein and Sloane \cite{bernstein1995} studied integer sequences that reproduce themselves, up to a shift, under transformations of this kind. The classical Bell numbers provide the best-known example:
\begin{equation}\label{eq-bell-recurrence}
B_{n+1}=\sum_{k=0}^n\binom{n}{k}B_k,
\qquad B_0=1.
\end{equation}
Equivalently, if $\Lshift(a_0,a_1,\ldots)=(a_1,a_2,\ldots)$, then $\BINOM B=\Lshift B$.

The Bell numbers count the total number of partitions of an $n$-element set, and they are part of a rich combinatorial structure that involves the Stirling numbers of the second kind (\href{https://oeis.org/A008277}{A008277}), Touchard polynomials \cite{weisstein}, the exponential function, and linear operators acting on it \cite{dattoliTouchardPolynomialsGeneralized2010}.

This paper studies the simultaneous $m$-fold generalization
\begin{equation}\label{eq:m-shift-recurrence}
a_{n+m}=\sum_{k=0}^n\binom{n}{k}m^{\,n-k}a_k,
\qquad n\geq0,
\end{equation}
or $\BINOM^m a=\Lshift^m a$. Equation \eqref{eq:m-shift-recurrence} is an order-$m$ recurrence: it defines an $m$-dimensional solution space, not a single sequence. We call its members \emph{$m$-binomial--shift sequences}. The solution with initial vector $(1,\ldots,1)$ will be called the \emph{canonical $m$-Bell sequence} and denoted $B^{(m)}=(B_n^{(m)})_{n\geq0}$.

Any such sequence is determined by its first $m$ values $(a_0,\ldots,a_{m-1})$, so the solution set is an $m$-dimensional lattice of integer sequences spanned by the $m$ \textbf{primitive} solutions $B^{(m,r)}$, $r = 0, \ldots, m-1$, whose initial conditions are $B^{(m,r)}_n = \delta_{n,r}$ for $0 \le n < m$. We reserve the unadorned symbol $B^{(m)} = B^{(m,0)} + \cdots + B^{(m,m-1)}$ for the \textbf{composite} sequence with initial conditions $(1,1,\ldots,1)$. The case $m=1$ recovers the Bell numbers, and $m=2$ corresponds to sequences \href{https://oeis.org/A007472}{A007472}, \href{https://oeis.org/A351143}{A351143} and \href{https://oeis.org/A351028}{A351028}, which shift by 2 places left after 2 binomial transforms (the sequences for $m>2$ are not currently present in the OEIS). Although the sequences for $m=2$ are listed in the OEIS, little was previously known about their properties.

In this paper we show that the $m$-fold binomial--shift-invariance property that characterizes these sequences arises from a combinatorial structure that closely mirrors that of the classical Bell numbers. The $m$-Bell numbers have exponential generating functions (e.g.f.s) that are the solutions of ordinary differential equations of order $m$. The solutions to these equations are a class of hypergeometric (``hyper-Bessel'') functions, which reduce to the exponential function for $m=1$ and to modified Bessel functions of the first and second kind for $m=2$, motivating the name Bessel--Bell numbers for the $2$-Bell case (Section \ref{sec-m2}; Theorem \ref{thm-m2-egf}). These generalized Bell numbers admit Dobi\'nski-like formulas (Theorem \ref{thm-bessel-dobiski} and Theorem \ref{thm-general-dobinski}). 

The main result is that each $m$-Bell sequence arises from a Stirling-like triangular array generated by the two-term recurrence
\[
\mStirling{n+1}{k}_m = m\floor{k/m}\,\mStirling{n}{k}_m + \mStirling{n}{k-1}_m ,
\]
motivating the name \emph{$m$-Stirling numbers} (Sections \ref{sec-2-stirling} and \ref{sec-general-case}). Our central structural result is an elementary shift identity for these triangles (Theorem \ref{thm-shift-identity}), from which we deduce that the row sums of the $m$-Stirling triangle form the composite $m$-Bell sequence and---more finely---that the row sums along residue classes of $k \bmod m$ are precisely the $m$ primitive $m$-Bell sequences (Corollary \ref{cor-mbell-rowsums}). For $m = 2$ this identifies A351143 and A351028 as the even-$k$ and odd-$k$ halves of a single triangle whose full row sums give A007472.

Like the classical Stirling numbers, the $m$-Stirling numbers come in dual pairs: there are companion triangles of the first kind, together with generalized falling factorials and Lah-type arrays, and the whole family acts as a system of conversion operators between polynomial bases (Section \ref{sec-first-kind}). Each $m$-Stirling array of the second kind also arises from the coefficients of polynomials generated by applying the exponential scaling operator to the hypergeometric e.g.f.s, once these are paired with appropriate \emph{carrier functions} (Sections \ref{sec-2-stirling} and \ref{sec-general-case}). On the combinatorial side, the $m$-Stirling and $m$-Bell numbers count congruence-constrained partitions in an urn model with $m$-compartment urns (Section \ref{sec-combinatorial}), while the first-kind numbers count permutations built by a parity-restricted insertion process (Section \ref{sec-first-kind}). 

Finally, just as the classical Bell numbers are the moments of the Poisson distribution with unit rate, a key result is that the $m$-Bell numbers govern the moments of a well-known generalization of the Poisson distribution, the Conway--Maxwell--Poisson distribution \cite{shmueliUsefulDistributionFitting2005}, at integer dispersion parameter $\nu = m$ (Section \ref{sec-cmp}). Appendix D collects the parallels between the classical framework and its generalization in a single reference table.

To clearly ground the analogy between the Bell--Stirling--Touchard framework and the novel results, we begin with a review of standard results and notation \cite{comtet1974, sándor2004}.

\section{Background}
\noindent The exponential generating function (e.g.f.) $\mathcal{A}(x)$ of a sequence $(a_n)_{n \geq 0}$ is a formal power series in $x$:
\[
\mathcal{A}(x) = \sum_{n=0}^\infty a_n \frac{x^n}{n!}.
\]
\begin{proposition}[Functional equation for the e.g.f.\ of $m$-Bell numbers]\label{prop:mbell-egf}
Let $\mathcal{A}(x)$ be the exponential generating function of a sequence \( a = (a_n) \) satisfying \eqref{eq:m-shift-recurrence}.
Then \( \mathcal{A}(x) \) satisfies the $m$-th order ODE
\[
\left(\frac{\mathrm{d}}{\mathrm{d}x}\right)^m \mathcal{A}(x) - e^{mx} \mathcal{A}(x) = 0.
\]
\end{proposition}

\begin{proof}
The binomial transform and the left-shift operator acting on $a$ have simple effects on the e.g.f., $\mathcal{B}(x)$, of the resulting sequence $b$ (e.g.\ \cite{bernstein1995}):
\begin{equation}
\begin{aligned}
b = \mathrm{BINOM} \circ a \quad &\iff\quad \mathcal{B}(x) = e^x \mathcal{A}(x) \\
b = \mathrm{L} \circ a \quad &\iff\quad \mathcal{B}(x) = \mathcal{A}'(x).
\end{aligned}
\end{equation}
Applying $\mathrm{BINOM}$ $m$ times corresponds to multiplying $\mathcal{A}(x)$ by $e^{mx}$. Applying $\mathrm{L}$ $m$ times corresponds to taking the $m$-th derivative. Therefore, if \(\mathrm{BINOM}^m \circ a = \mathrm{L}^m \circ a\), then
\[
\mathcal{A}^{(m)}(x) = e^{mx} \mathcal{A}(x).
\]
Rewriting this gives the result.
\end{proof}

For $m=1$ the ODE $\mathcal{A}'(x) = e^{x}\mathcal{A}(x)$ with $\mathcal{A}(0)=1$ integrates at once to the celebrated e.g.f.\ of the Bell numbers:
\begin{equation}\label{eq-bell-egf}
    e^{e^x-1} = \sum_{n=0}^{\infty} B_n \frac{x^n}{n!}.
\end{equation}
Recall that the Bell numbers are the row sums of the triangular array formed by the Stirling numbers of the second kind:
\begin{align*}
B_n = \sum_{k=0}^{n}\genfrac\{\}{0pt}{}{n}{k},
\end{align*}
where the Stirling numbers of the second kind satisfy the two-term recurrence
\[
\genfrac\{\}{0pt}{}{n+1}{k} = k \, \genfrac\{\}{0pt}{}{n}{k} + \genfrac\{\}{0pt}{}{n}{k-1}, \,\,\,\,\, n, k \ge 1.
\]
A second classical recurrence, which will serve as the model for our main structural result (Theorem \ref{thm-shift-identity}), expresses a downward shift of the whole triangle through binomial averaging \cite{comtet1974, graham_concrete_nodate}:
\begin{equation}\label{eq-stirling-shift}
\Stirling{n+1}{k} = \sum_{j=0}^{n}\binom{n}{j}\Stirling{j}{k-1}.
\end{equation}
Summing \eqref{eq-stirling-shift} over $k$ recovers the Bell recurrence \eqref{eq-bell-recurrence}; combinatorially, the binomial coefficient chooses which elements avoid the block containing the largest element.

The Stirling numbers of the second kind count the number of ways to partition $n$ labeled objects into $k$ unlabeled non-empty subsets and are the coefficients of the Touchard (also known as Bell or exponential) polynomials:
\begin{equation}
    T_n(x) = \sum_{k=0}^{n} \genfrac\{\}{0pt}{}{n}{k} x^k,
\end{equation}
whose e.g.f.\ is a bivariate generalization of the Bell numbers' e.g.f.:
\begin{equation}
    \sum_{n=0}^{\infty} T_n(x) \frac{t^n}{n!} = e^{x(e^t-1)}.
\end{equation}
Many useful identities involving Touchard polynomials and Stirling numbers can be shown via the action of the exponential scaling operator
\begin{equation}\label{eq-scaling-operator}
    e^{txD_x}f(x)=f(xe^t),
\end{equation}
where $D_x$ is the derivative operator with respect to $x$. Specifically, by applying this operator to the standard exponential function $e^x$ we get precisely the e.g.f.\ for the Touchard polynomials \cite{dattoliTouchardPolynomialsGeneralized2010}:
\[
e^{-x}e^{txD_x}e^x = e^{x(e^t-1)}.
\]
A final useful relation is the celebrated Dobi\'nski formula \cite{wilfGeneratingfunctionologyThirdEdition2005} that lets us express $T_n(x)$ and the Bell numbers $B_n = T_n(1)$ as an infinite sum:
\begin{equation}\label{eq-bell-dobinski}
\begin{aligned}
T_n(x) & =e^{-x}\sum_{k=0}^{\infty}\frac{x^k k^n}{k!} \\
B_n & = \frac{1}{e}\sum_{k=0}^{\infty}\frac{k^n}{k!}.
\end{aligned}
\end{equation}
In the remainder of the paper we will see that the $m$-Bell sequences have properties analogous to all of those defined above, starting with the special case $m = 2$ which motivated this exploration.

\section{The Bessel--Bell numbers ($m=2$)}\label{sec-m2}
\subsection{Deriving the e.g.f.}
\noindent To build intuition, before solving the general case, let us focus on the $m=2$ generalization of the Bell numbers, which satisfy the property that they shift by 2 places left after two binomial transforms:
\begin{equation}\label{eq-2bell-recurrence}
B^{(2)}_{n+2} = \sum_{k=0}^n \binom{n}{k} 2^{n-k} B_k^{(2)}.
\end{equation}
The following three sequences listed in the OEIS satisfy this property:
\begin{itemize}
    \item \href{https://oeis.org/A007472}{A007472}: $1, 1, 1, 3, 9, 29, 105, 431, 1969, \ldots$
    \item \href{https://oeis.org/A351143}{A351143}: $1, 0, 1, 2, 5, 16, 61, 258, 1177, \ldots$
    \item \href{https://oeis.org/A351028}{A351028}: $0, 1, 0, 1, 4, 13, 44, 173, 792, 4009, \ldots$
\end{itemize}
The only difference between the three sequences is the pair of initial conditions $(a_0, a_1)$: $(1, 1)$, $(1,0)$ and $(0,1)$, respectively. In the terminology of Section \ref{sec-introduction}, the primitive solutions for $m=2$ are A351143 $= B^{(2,0)}$ and A351028 $= B^{(2,1)}$, whose element-wise sum produces the composite sequence A007472 $= B^{(2)}$. Owing to the form of their e.g.f.s, we will refer to the sequences satisfying \eqref{eq-2bell-recurrence} as Bessel--Bell numbers (not to be confused with the Bessel numbers of \cite{cheonGeneralizedBesselNumbers2013a}, which arise from a different structure).

\begin{theorem}\label{thm-m2-egf}
The exponential generating functions of the $2$-Bell sequences are solutions of the modified Bessel ODE composed with $e^x$: each is a linear combination
\[
\mathcal{A}(x) = p\,I_0(e^x)+q\,K_0(e^x)
\]
of modified Bessel functions of the first ($I_0$) and second kind ($K_0$) of order 0, whose weights $p, q$ are uniquely determined by the first two elements of the corresponding sequence. Specifically:

\begin{center}
\begin{tabular}{lrr}
\hline
Sequence & Coefficient \(p\) & Coefficient \(q\) \\
\hline
A351143 & \(K_1(1) \approx 0.601907\) & \(I_1(1) \approx 0.565159\) \\
A351028 & \(K_0(1) \approx 0.421024\) & \(-I_0(1) \approx -1.266066\) \\
A007472 & \(K_0(1) + K_1(1)\approx 1.022932\) & \(I_1(1) - I_0(1)\approx -0.700907\) \\
\hline
\end{tabular}
\end{center}
\smallskip
\end{theorem}

\begin{proof}
By Proposition \ref{prop:mbell-egf}, the e.g.f.\ $\mathcal{A}(x)$ satisfies the linear ODE
\[
\mathcal{A}''(x)-e^{2x}\,\mathcal{A}(x)=0.
\]
Set $t=e^{x}$ (so $\mathrm{d}t/\mathrm{d}x=t$) and define $f(t)=\mathcal{A}(x)$.  By the chain rule,
\[
\begin{aligned}
\mathcal{A}'(x) &= t\,f'(t),\\
\mathcal{A}''(x)&= t\,f'(t)+t^{2}f''(t),\\
e^{2x}\,\mathcal{A}(x) &= t^{2}f(t).
\end{aligned}
\]
Substituting in yields the standard form of the modified Bessel equation of order $0$ (see \href{https://dlmf.nist.gov/10.25.E1}{10.25.1}) \cite{NIST:DLMF}
\begin{equation}\label{eq-m2-ode}
t^{2}f''+t\,f'-t^{2}f=0,
\end{equation}
whose general solution is
\[
f(t)=p\,I_0(t)+q\,K_0(t),
\]
with $I_\nu$ and $K_\nu$ being the modified Bessel functions of the first and second kind, and $p$ and $q$ constants determined by the initial conditions. Therefore, we obtain the general solution for $\mathcal{A}(x)$, the e.g.f.\ of sequences \href{https://oeis.org/A007472}{A007472}, \href{https://oeis.org/A351143}{A351143} and \href{https://oeis.org/A351028}{A351028}, by substituting back $t=e^x$:
\begin{equation}
\mathcal{A}(x)=p\, I_0(e^x) + q\,K_0(e^x).
\end{equation}
To determine $p$ and $q$ we use the known initial conditions. Sequence A007472 is the element-wise sum of A351143 and A351028, so it is sufficient to determine $p$ and $q$ for the latter two sequences only. For A351143 we have $a_0 = 1$ and $a_1 = 0$. Matching the first two Maclaurin coefficients of $\mathcal{A}(x)$ yields the following system of equations:
\[
\begin{aligned}
\mathcal{A}(0) & = p\, I_0(1) + q\,K_0(1) = 1 \\
\mathcal{A}'(x)\bigr|_{x=0} &=p\, \tfrac{\mathrm{d}}{\mathrm{d}x}I_0(e^x)\bigr|_{x=0} + q\,\tfrac{\mathrm{d}}{\mathrm{d}x}K_0(e^x)\bigr|_{x=0} = 0,
\end{aligned}
\]
where $f(x)\bigr|_{x=0}$ denotes evaluation at 0. The derivatives of $I_0$ and $K_0$ are thankfully straightforward (\href{https://dlmf.nist.gov/10.29.E3}{10.29.3}) \cite{NIST:DLMF}:
\begin{equation}
\begin{aligned}
I_0'(z) &= I_1(z) \\
K_0'(z) &= -K_1(z). 
\end{aligned}
\end{equation}
By the chain rule we get
\begin{equation*}
    \begin{aligned}
        \tfrac{\mathrm{d}}{\mathrm{d}x}I_0\bigl(e^x\bigr) &= e^xI_1\bigl(e^x\bigr) \\
        \tfrac{\mathrm{d}}{\mathrm{d}x}K_0\bigl(e^x\bigr) &= -e^xK_1\bigl(e^x\bigr). 
    \end{aligned}
\end{equation*}
Evaluating at $x=0$ gives
\[
\begin{aligned}
\mathcal{A}(0) & = p \,I_0(1)+q\,K_0(1) = 1 \\
\mathcal{A}'(0) & = p\,I_1(1)-q K_1(1) = 0.
\end{aligned}
\]
After some standard algebraic torture we get the following expressions for $p$ and $q$:
\[
\begin{aligned}
p &= \frac{1-qK_0(1)}{I_0(1)} \\
q &= \frac{I_1(1)}{I_1(1)K_0(1)+I_0(1)K_1(1)}.
\end{aligned}
\]
These expressions can be simplified further due to the following Bessel identity concerning the Wronskian of the modified Bessel functions (see \href{https://dlmf.nist.gov/10.28.E2}{10.28.2}) \cite{NIST:DLMF}:
\begin{equation}\label{eq-bessel-wronskian-general}
    I_\nu(z)K_{\nu+1}(z)+I_{\nu+1}(z)K_\nu(z) = 1/z,
\end{equation}
which holds for every complex $\nu$ and every $z \neq 0$ (with the principal branches of $I_\nu$ and $K_\nu$; we only ever use it at real $z > 0$). In the special case when $\nu = 0$ and $z=1$:
\begin{equation}\label{eq-bessel-wronskian}
I_0(1)K_1(1)+I_1(1)K_0(1) = 1.
\end{equation}
Therefore the constants for the e.g.f.\ of sequence A351143 are
\[
\begin{aligned}
q &= I_1(1) \\
p &= \frac{1 - I_1(1)K_0(1)}{I_0(1)} = \frac{I_0(1)K_1(1)+I_1(1)K_0(1) - I_1(1)K_0(1)}{I_0(1)} = K_1(1).
\end{aligned}
\]
A very similar manipulation for the initial conditions $(0,1)$ of A351028 gives $p = K_0(1)$ and $q = -I_0(1)$, and adding the two solutions gives the coefficients for A007472 stated in the table, which concludes the proof.
\end{proof}

\subsection{Why do the e.g.f.s of the 2-Bell numbers produce integer coefficients?}\label{2-bell-integer-coefs}
The e.g.f.s of the 2-Bell numbers are admittedly more complicated than the e.g.f.\ $e^{e^x-1}$ of the Bell numbers \eqref{eq-bell-egf} (although the form is quite similar, as we discuss at the end of this section). How is it that such a relatively complicated expression, involving four different special functions, produces integer coefficients? Let us explore the series expansion of the derived e.g.f.s. Consider first the case of sequence A351143:
\[
\mathcal{A}_{351143}(x) = K_1(1)I_0(e^x) + I_1(1)K_0(e^x).
\]
We have the following standard series for $I_0$ and $K_0$, valid for real $z > 0$, which is the only case we need (with the principal branch of the logarithm they extend to the cut plane; see \href{https://dlmf.nist.gov/10.25.E2}{10.25.2} and \href{https://dlmf.nist.gov/10.31.E2}{10.31.2}) \cite{NIST:DLMF}:
\begin{equation}
    \begin{aligned}
        I_0(z) &= \sum_{k=0}^{\infty} \frac{(\frac{1}{2}z)^{2k}}{k!k!} \\
        K_0(z) &= -\log\biggl(\frac{z}{2}\biggr)I_0(z) + \sum_{k=0}^{\infty}\frac{\psi(k+1)(\frac{1}{2}z)^{2k}}{k!k!}
    \end{aligned}
\end{equation}
where $\psi$ is the digamma function. Let us focus on the $I_0$ function. Substituting $z=e^x$ and then expanding the Taylor series for the exponential function we get
\begin{equation}\label{eq-besseli-exp-maclaurin}
\begin{aligned}
I_0(e^x) &= \sum_{k=0}^{\infty} \frac{e^{2xk}}{2^{2k}k!k!} = \sum_{k=0}^{\infty} \frac{1}{2^{2k}k!k!}\sum_{n=0}^{\infty}\frac{x^n(2k)^n}{n!} = \sum_{n=0}^{\infty}\frac{x^n}{n!}\sum_{k=0}^{\infty} \frac{(2k)^n}{2^{2k}k!k!}.
\end{aligned}
\end{equation}
(Interchanging the order of summation is justified by absolute convergence.) Let the inner sum be represented by $S(n) = \sum_{k=0}^{\infty} \frac{(2k)^n}{2^{2k}k!k!}$, with the convention $0^0 = 1$ here and in every Dobi\'nski-type sum below (so the $k = 0$ term contributes exactly $\delta_{n,0}$). This formula is reminiscent of the Dobi\'nski formula for the standard Bell numbers \eqref{eq-bell-dobinski}, with an extra factorial in the denominator and extra powers of 2. Indeed, we can state an equivalent theorem:

\begin{theorem}[Dobi\'nski-like formula for the Bessel--Bell numbers]\label{thm-bessel-dobiski}
Let $v_n = B^{(2,0)}_n$ and $u_n = B^{(2,1)}_n$ denote the primitive $2$-Bell sequences A351143 and A351028. Then
\[
\begin{aligned}
I_0(e^x) &= \sum_{n=0}^\infty S(n) \frac{x^n}{n!}, \qquad \text{where}\\
S(n) &= \sum_{k=0}^{\infty} \frac{(2k)^n}{(2^{k}k!)^2}=v_n I_0(1) + u_n I_1(1).
\end{aligned}
\]
\end{theorem}

\begin{proof}
Proceed by strong induction. First, establish the base cases:
\[
\begin{aligned}
S(0) &= \sum_{k=0}^{\infty} \frac{1}{2^{2k}k!k!} = I_0(1) \\
S(1) &= \sum_{k=0}^{\infty} \frac{2k}{2^{2k}k!k!} = \sum_{k=1}^{\infty} \frac{1}{2^{2k-1}(k-1)!k!} = \sum_{k=0}^{\infty} \frac{1}{2^{2k+1}k!(k+1)!} = I_1(1),
\end{aligned}
\]
where in the case of $S(1)$ we used the series expansion of the modified Bessel function of general integer order $\nu \ge 0$ evaluated at $1$ (\href{https://dlmf.nist.gov/10.25.E2}{10.25.2}) \cite{NIST:DLMF}:
\[
I_\nu(z) = \sum_{k=0}^\infty \frac{(z/2)^{2k+\nu}}{k!\,(k+\nu)!}.
\]
These match the claim since $(v_0, u_0) = (1,0)$ and $(v_1, u_1) = (0,1)$. For the inductive step, suppose that $S(j) = v_j I_0(1) + u_j I_1(1)$ holds for every $j \leq n+1$. Then:
\begin{equation}\label{eq-2bell-dobinski-proof-step1}
\begin{aligned}
S(n+2) &= \sum_{k=0}^{\infty} \frac{(2k)^{n+2}}{2^{2k}k!k!} = \sum_{k=1}^{\infty} \frac{(2k)^{n}}{2^{2k-2}(k-1)!(k-1)!} = \sum_{k=0}^{\infty} \frac{(2k+2)^{n}}{2^{2k}k!k!}\\
&= \sum_{k=0}^{\infty}\sum_{j=0}^{n} \binom{n}{j}2^{n-j}\frac{(2k)^j}{2^{2k}k!k!} = \sum_{j=0}^{n}\binom{n}{j}2^{n-j} \sum_{k=0}^{\infty}\frac{(2k)^j}{2^{2k}k!k!} \\
&= \sum_{j=0}^{n}\binom{n}{j}2^{n-j} S(j).
\end{aligned}
\end{equation}
First, notice that \eqref{eq-2bell-dobinski-proof-step1} has exactly the same form as the recurrence relation \eqref{eq-2bell-recurrence} that defines the 2-Bell numbers. Second, use the induction hypothesis and substitute $S(j)$:
\[
\begin{aligned}
S(n+2) &= \sum_{j=0}^{n}\binom{n}{j}2^{n-j} (v_j I_0(1)+u_j I_1(1)) \\
&= I_0(1)\bigg[\sum_{j=0}^{n}\binom{n}{j}2^{n-j}v_j\bigg] + I_1(1)\bigg[\sum_{j=0}^{n}\binom{n}{j}2^{n-j}u_j\bigg] \\
&= v_{n+2}I_0(1)+u_{n+2}I_1(1),
\end{aligned}
\]
where the last step applies the defining recurrence \eqref{eq-2bell-recurrence} of the primitive sequences $v$ and $u$. This completes the proof. Table \ref{tab-vn-un-coefs} illustrates the first few cases.
\end{proof}

\begin{table}[htbp]
    \centering
        \begin{tabular}{|c|l|c|c|c|}
        \hline
        $n$ & $S(n)$                & $v_n$ & $u_n$ & $v_n+u_n$ \\
        \hline
        0   & $I_0(1)$              & 1     & 0     & 1         \\
        1   & $I_1(1)$              & 0     & 1     & 1         \\
        2   & $I_0(1)$              & 1     & 0     & 1         \\
        3   & $2I_0(1)+I_1(1)$    & 2     & 1     & 3         \\
        4   & $5I_0(1)+4I_1(1)$   & 5     & 4     & 9         \\
        5   & $16I_0(1)+13I_1(1)$ & 16    & 13    & 29        \\
        6   & $61I_0(1)+44I_1(1)$ & 61    & 44 &    105 \\
        \hline
        \end{tabular}
        \bigskip
    \caption{Coefficients in the expansion of the Dobi\'nski-like formula for the Bessel--Bell numbers (Theorem \ref{thm-bessel-dobiski}).}
    \label{tab-vn-un-coefs}
\end{table}

The major difference between the Dobi\'nski-like formula of Theorem \ref{thm-bessel-dobiski} and the standard Dobi\'nski formula \eqref{eq-bell-dobinski} is that there are two different modified Bessel functions acting as \textbf{carriers} for the coefficients of the two primitive Bessel--Bell sequences, rather than a single exponential function. In the standard Dobi\'nski formula one multiplies the infinite sum by $1/e$ to cancel the exponential carrier; here no single multiplicative factor can cancel two independent carriers at once. The concept of a carrier function will play a fundamental role in understanding how $m$-Bell and $m$-Stirling numbers are generated, and we formalize it in Section \ref{sec-2-stirling} (Definition \ref{def-carrier-system}).

What about the second Bessel function $K_0(e^x)$ appearing in the e.g.f.s? Its Maclaurin coefficients turn out to involve exactly the same integer sequences, with one sign change.

\begin{theorem}\label{thm-besselk-dobinski}
With $v_n$ and $u_n$ as in Theorem \ref{thm-bessel-dobiski},
\[
K_0(e^x) = \sum_{n=0}^\infty \frac{x^n}{n!}\big(v_n K_0(1) - u_n K_1(1)\big).
\]
\end{theorem}

\begin{proof}
Consider the two pairs of functions
\[
F_0(x) = I_0(e^x), \quad F_1(x) = I_1(e^x), \qquad G_0(x) = K_0(e^x), \quad G_1(x) = -K_1(e^x).
\]
From the derivative formulas $I_0' = I_1$, $I_1'(z) = I_0(z) - z^{-1}I_1(z)$, $K_0' = -K_1$ and $K_1'(z) = -K_0(z) - z^{-1}K_1(z)$ (\href{https://dlmf.nist.gov/10.29.E2}{10.29.2}) \cite{NIST:DLMF}, the chain rule gives
\[
\begin{aligned}
F_0' &= e^x F_1, & F_1' &= e^x F_0 - F_1,\\
G_0' &= e^x G_1, & G_1' &= e^x G_0 - G_1.
\end{aligned}
\]
That is, $(F_0, F_1)$ and $(G_0, G_1)$ satisfy the \emph{same} first-order linear system
\begin{equation}\label{eq-first-order-system}
X' = e^x\,Y, \qquad Y' = e^x\,X - Y.
\end{equation}
Define coefficient functions $a_n, b_n \in \mathbb{Z}[e^x]$ by the recursion
\[
a_{n+1} = a_n' + e^x b_n, \qquad b_{n+1} = e^x a_n + b_n' - b_n, \qquad a_0 = 1,\ b_0 = 0.
\]
A one-line induction shows that \emph{every} solution pair $(X, Y)$ of \eqref{eq-first-order-system} satisfies
\begin{equation}\label{eq-universal-expansion}
X^{(n)} = a_n(x)\,X + b_n(x)\,Y \qquad (n \geq 0),
\end{equation}
since differentiating \eqref{eq-universal-expansion} and substituting \eqref{eq-first-order-system} reproduces the recursion. We claim that $a_n(0) = v_n$ and $b_n(0) = u_n$. To see this, first note that any solution of \eqref{eq-first-order-system} has first component satisfying the 2-Bell ODE: $X'' = (e^xY)' = e^xY + e^xY' = X' + e^{2x}X - e^xY = e^{2x}X$, with $X'(0) = Y(0)$. Comparing Maclaurin coefficients on both sides of $X'' = e^{2x}X$ shows that the sequence $\big(X^{(n)}(0)\big)_n$ satisfies the 2-Bell recurrence \eqref{eq-2bell-recurrence}. Now take the solution pair with initial data $(X(0), Y(0)) = (1, 0)$: then $X(0) = 1$, $X'(0) = 0$, so $\big(X^{(n)}(0)\big)_n$ is precisely the primitive sequence A351143, i.e.\ $X^{(n)}(0) = v_n$; but \eqref{eq-universal-expansion} evaluated at $0$ gives $X^{(n)}(0) = a_n(0)$. Hence $a_n(0) = v_n$, and the solution pair with initial data $(0,1)$ gives $b_n(0) = u_n$ in the same way. Applying \eqref{eq-universal-expansion} to the pair $(G_0, G_1)$ and evaluating at $x = 0$ yields
\[
G_0^{(n)}(0) = a_n(0)\, K_0(1) + b_n(0)\, \big({-}K_1(1)\big) = v_n K_0(1) - u_n K_1(1). \qedhere
\]
\end{proof}

\begin{remark}
The same evaluation applied to the pair $(F_0, F_1)$ re-derives Theorem \ref{thm-bessel-dobiski}, since $F_0^{(n)}(0) = S(n)$ by \eqref{eq-besseli-exp-maclaurin}. The two proofs are complementary: the induction of Theorem \ref{thm-bessel-dobiski} works directly with the series, while the present argument isolates the structural reason---$I$- and $K$-type functions solve the same differential system---why the same integer sequences appear in both expansions.
\end{remark}

Finally, by combining Theorem~\ref{thm-bessel-dobiski} and Theorem~\ref{thm-besselk-dobinski}, we have
\begin{equation}
\begin{aligned}
p\, I_0(e^x) + q\, K_0(e^x) = \sum_{n=0}^\infty \frac{x^n}{n!} \bigg[ v_n \big(p I_0(1) + q K_0(1)\big) + u_n \big(p I_1(1) - q K_1(1)\big)\bigg].
\end{aligned}
\end{equation}
Within this equation, an appropriate choice of $p$ and $q$ neutralizes the modified Bessel functions in the series thanks to the Wronskian identity \eqref{eq-bessel-wronskian}. Specifically, $p = K_1(1), q = I_1(1)$ makes the coefficient of $v_n$ equal to 1 and the coefficient of $u_n$ equal to 0, giving integer coefficients---the series A351143. Similarly for the other two choices of $p$ and $q$ established in Theorem~\ref{thm-m2-egf}: $(K_0(1), -I_0(1))$ isolates $u_n$, and the sum of the two choices produces $v_n + u_n$.

Thus the e.g.f.\ of a 2-Bell sequence is structurally similar to that of the 1-Bell (Bell) sequence: it is the composition of an exponential-like function with the exponential function, and the remaining constant factors are there just to ensure integer coefficients, exactly like the factor $e^{-1}$ in $e^{-1}e^{e^x}$. With the integer-coefficient property now rigorously settled, let us shift to the combinatorial side: a triangular recurrence whose row sums reproduce precisely the Bessel--Bell sequences.

\section{Bessel--Stirling numbers of the second kind}\label{sec-2-stirling}

\noindent The similarity of the Bessel--Bell numbers' e.g.f.\ to the standard Bell numbers' e.g.f.\ immediately invites the question of whether there is also a Bessel--Stirling equivalent of the triangular array of the Stirling numbers of the second kind, whose rows sum to the Bell numbers. There is, and the correct recurrence modifies the classical one in a single place.

\begin{definition}[Bessel--Stirling numbers of the second kind]
Let $\mStirling{n}{k}_2$ be defined for integers $n, k \geq 0$ by the initial conditions $\mStirling{0}{k}_2 = \delta_{0,k}$, the convention $\mStirling{n}{k}_2 = 0$ for $k < 0$, and the two-term recurrence
\begin{equation}\label{eq-bs-recurrence}
\mStirling{n+1}{k}_2 = 2\floor{k/2} \, \mStirling{n}{k}_2 + \mStirling{n}{k-1}_2.
\end{equation}
The notation follows the bracket style for binomial coefficients $\binom{n}{k}$ and Stirling numbers $\Stirling{n}{k}$ popularized by \cite{graham_concrete_nodate}, with doubled delimiters and a subscript indicating the modulus (here $2$); Section \ref{sec-general-case} treats general modulus $m$.
\end{definition}

This recurrence is almost identical to that of the Stirling numbers of the second kind, except that the coefficient in front of $\mStirling{n}{k}_2$ has parity---it is always even, with $k$ rounded down to the nearest even integer. The numbers $\mStirling{n}{k}_2$ form a triangular array whose first few rows are listed in Table \ref{tab-bessel-stirling2}; the triangle is available as OEIS entry \href{https://oeis.org/A383235}{A383235}, which was submitted in the course of this work. An immediate consequence of the definition is that $\mStirling{n}{k}_2 = 0$ for $k > n$ (each application of \eqref{eq-bs-recurrence} raises the maximal index $k$ by one), that $\mStirling{n}{n}_2 = 1$, and that $\mStirling{n}{1}_2 = \delta_{n,1}$, since the coefficient $2\floor{k/2}$ vanishes for $k \in \{0, 1\}$.

\begin{table}[htbp]
    \centering
        \begin{tabular}{|l|*{10}{c}|c|}
        \hline
              & k=0 & k=1 & k=2 & k=3 & k=4 & k=5 & k=6 & k=7 & k=8 & k=9 & $\sum_{k=0}^{n}$ \\
        \hline
        n=0 & 1   &     &     &     &     &     &     &     &     &     & 1                \\
        n=1 & 0   & 1   &     &     &     &     &     &     &     &     & 1                \\
        n=2 & 0   & 0   & 1   &     &     &     &     &     &     &     & 1                \\
        n=3 & 0   & 0   & 2   & 1   &     &     &     &     &     &     & 3                \\
        n=4 & 0   & 0   & 4   & 4   & 1   &     &     &     &     &     & 9                \\
        n=5 & 0   & 0   & 8   & 12  & 8   & 1   &     &     &     &     & 29               \\
        n=6 & 0   & 0   & 16  & 32  & 44  & 12  & 1   &     &     &     & 105              \\
        n=7 & 0   & 0   & 32  & 80  & 208 & 92  & 18  & 1   &     &     & 431              \\
        n=8 & 0   & 0   & 64  & 192 & 912 & 576 & 200 & 24  & 1   &     & 1969             \\
        n=9 & 0   & 0   & 128 & 448 & 3840& 3216& 1776& 344 & 32  & 1   & 9785             \\
        \hline
        \end{tabular}
        \bigskip
    \caption{The Bessel--Stirling numbers of the second kind, $\mStirling{n}{k}_2$ (OEIS \href{https://oeis.org/A383235}{A383235}). The values in the upper triangle are all 0.}
    \label{tab-bessel-stirling2}
\end{table}

The row sums of this array visibly reproduce the Bessel--Bell numbers \href{https://oeis.org/A007472}{A007472}:
\[
\sum_{k=0}^{n} \mStirling{n}{k}_2 = (1, 1, 1, 3, 9, 29, 105, 431, 1969, 9785, 52145, 296155, 1787385, \ldots).
\]
In this section we prove this and more via the operational calculus of \emph{carrier systems}; Section \ref{sec-general-case} gives an independent, completely elementary proof valid for every $m$.

\subsection{Carrier systems and Bessel--Touchard polynomials}

Recall from \eqref{eq-scaling-operator} that Touchard polynomials arise by applying the exponential scaling operator $e^{txD}$ to the carrier $e^x$, whose salient property is that the operator $xD$ maps the module $\{P(x)e^x : P \text{ polynomial}\}$ to itself. Modified Bessel functions are not preserved individually by $xD$, but the \emph{pair} $(I_0, I_1)$ is, in the following precise sense. From the derivative relations (\href{https://dlmf.nist.gov/10.29.E2}{10.29.2}) \cite{NIST:DLMF}:
\begin{equation}\label{eq-xd-bessel}
    \begin{aligned}
        (xD)I_0(x) &= x I_1(x) \\
        (xD)I_1(x) &= x I_0(x)-I_1(x).
    \end{aligned}
\end{equation}
We abstract exactly this structure:

\begin{definition}[Carrier system]\label{def-carrier-system}
Let $m \geq 1$. A \emph{carrier system of modulus $m$} for the operator $xD$ is a family $(f_0, f_1, \ldots, f_{m-1})$ of differentiable functions on a common open subset of $(0,\infty)$ (or of formal power series) satisfying the cyclic first-order system
\begin{equation}\label{eq-carrier-system}
(xD + r)\,f_r = x\,f_{(r+1) \bmod m}, \qquad r = 0, 1, \ldots, m-1.
\end{equation}
The function $f_0$ is called the \emph{primary carrier} of the system.
\end{definition}

For $m = 1$ the single condition reads $(xD)f_0 = x f_0$, i.e.\ $f_0' = f_0$, whose solutions are the constant multiples of $e^x$; we take the normalized solution $e^x$, the exponential carrier of the classical theory. For $m = 2$, equations \eqref{eq-xd-bessel} say precisely that $(I_0, I_1)$ is a carrier system: $(xD + 0)I_0 = xI_1$ and $(xD + 1)I_1 = xI_0$. The computation in the proof of Theorem \ref{thm-besselk-dobinski} shows that $(K_0, -K_1)$ is a second, independent carrier system of modulus 2. Chaining the $m$ conditions of \eqref{eq-carrier-system} shows that a primary carrier always satisfies the eigenvalue-like equation
\begin{equation}\label{eq-carrier-ode}
(xD)^m f_0 = x^m f_0,
\end{equation}
which for $m=2$ is the modified Bessel equation \eqref{eq-m2-ode}---the same equation, we emphasize, that governs the e.g.f.s of the 2-Bell numbers under the substitution $t = e^x$.

The next theorem shows that iterating $xD$ on a primary carrier manufactures the Bessel--Stirling triangle, in exactly the way that iterating $xD$ on $e^x$ manufactures the Stirling triangle. We state it for modulus 2; the statement and proof for general $m$ are identical (Theorem \ref{thm-general-carrier-expansion}).

\begin{theorem}[Carrier expansion]\label{thm-carrier-expansion}
Let $(f_0, f_1)$ be a carrier system of modulus 2 for $xD$. Then for all $n \geq 0$,
\[
(xD)^n f_0(x) \;=\; \sum_{k=0}^{n} \mStirling{n}{k}_2\, x^k\, f_{k \bmod 2}(x),
\]
and consequently, by the scaling operator identity \eqref{eq-scaling-operator},
\[
f_0(x e^t) = e^{txD}f_0(x) = \sum_{n=0}^\infty\frac{t^n}{n!}\, V_n(x), \qquad V_n(x) = \sum_{k=0}^{n} \mStirling{n}{k}_2\, x^k\, f_{k \bmod 2}(x).
\]
\end{theorem}

The first display is a finite algebraic identity, valid verbatim for formal power series. The second is an identity of formal power series in $t$; when $f_0$ is analytic on an open set containing $\{xe^s : \lvert s \rvert \leq \lvert t \rvert\}$, it holds moreover as an equality of analytic functions, by Taylor's theorem applied to $t \mapsto f_0(xe^t)$. All carrier systems used in this paper---$(I_0, I_1)$, $(K_0, -K_1)$, and the hyper-Bessel systems of Section \ref{sec-general-case}---consist of functions analytic on $(0,\infty)$, so for them the expansion is valid for all $x > 0$ and all real $t$.

\begin{proof}
Induction on $n$. The case $n = 0$ is the initial condition $\mStirling{0}{k}_2 = \delta_{0,k}$. For the step, apply $xD$ to a single term $x^k f_r$ with $r = k \bmod 2$. Since $xD$ is a derivation and $(xD)x^k = kx^k$,
\[
\begin{aligned}
(xD)\big(x^k f_r\big) &= k\,x^k f_r + x^k (xD)f_r = k\,x^k f_r + x^k\big(x f_{(r+1) \bmod 2} - r f_r\big) \\
&= (k - r)\,x^k f_r + x^{k+1} f_{(k+1) \bmod 2}.
\end{aligned}
\]
Because $r = k \bmod 2$, the factor $k - r$ equals $2\floor{k/2}$. Summing over $k$ with weights $\mStirling{n}{k}_2$ and collecting the coefficient of $x^k f_{k \bmod 2}$ yields exactly the recurrence \eqref{eq-bs-recurrence}, i.e.\ the coefficient of $x^kf_{k\bmod 2}$ in $(xD)^{n+1}f_0$ is $2\floor{k/2} \mStirling{n}{k}_2 + \mStirling{n}{k-1}_2 = \mStirling{n+1}{k}_2$.
\end{proof}

Taking $(f_0, f_1) = (I_0, I_1)$, the first few polynomials $V_n$ are:
\[
\begin{aligned}
V_0(x) &= I_0(x) \\
V_1(x) &= x\,I_1(x) \\
V_2(x) &= x^2\,I_0(x) \\
V_3(x) &= 2x^2\,I_0(x) + x^3\,I_1(x) \\
V_4(x) &= 4x^2\,I_0(x) + 4x^3\,I_1(x) + x^4\,I_0(x) \\
V_5(x) &= 8x^2\,I_0(x) + 12x^3\,I_1(x) + 8x^4\,I_0(x) + x^5\,I_1(x).
\end{aligned}
\]
Notice two things. First, these polynomials are generated in very much the same way as the Touchard polynomials, but by applying the exponential scaling operator to $I_0(x)$ instead of to $e^x$. Second, each monomial $x^k$ is accompanied by a modified Bessel carrier function: $I_0$ ``carries'' the even degrees of $x$ and $I_1$ ``carries'' the odd degrees. The proper analogue of the Touchard polynomials should not have these carrier functions attached. Thankfully, the solution is the same one that produced integer coefficients for the e.g.f.s in Section \ref{2-bell-integer-coefs}: combine the two independent carrier systems, with weights chosen so that the general Wronskian identity \eqref{eq-bessel-wronskian-general} cancels all carriers.

\begin{theorem}[Bessel--Touchard polynomials]\label{thm-bb-polynomials}
Define the bivariate generating function
\[
G(t,x) = x\big[K_0(x)+K_1(x)\big]\, I_0(x e^t) + x\big[I_1(x)-I_0(x)\big]\,K_0(xe^t).
\]
Then, for $x > 0$,
\begin{equation*}
G(t,x) = \sum_{n=0}^\infty \frac{t^n}{n!}\,\mathscr{B}_n(x), \qquad \mathscr{B}_n(x) = \sum_{k=0}^{n}\mStirling{n}{k}_2\, x^k .
\end{equation*}
We call $\mathscr{B}_n(x)$ the Bessel--Touchard (or Bessel--Bell) polynomials.
\end{theorem}

\begin{proof}
Apply Theorem \ref{thm-carrier-expansion} to the two carrier systems $(I_0, I_1)$ and $(K_0, -K_1)$:
\[
\begin{aligned}
I_0(xe^t) &= \sum_{n=0}^\infty\frac{t^n}{n!}\sum_{k=0}^n \mStirling{n}{k}_2 x^k\, I_{k \bmod 2}(x), \\
K_0(xe^t) &= \sum_{n=0}^\infty\frac{t^n}{n!}\sum_{k=0}^n \mStirling{n}{k}_2 x^k\, \varkappa_k(x), \qquad \varkappa_k = \begin{cases} K_0 & k \text{ even}\\ -K_1 & k \text{ odd.}\end{cases}
\end{aligned}
\]
Now form the combination $p(x)\,I_0(xe^t) + q(x)\,K_0(xe^t)$ with as yet undetermined weights. The coefficient of $\mStirling{n}{k}_2 x^k$ in the $t$-expansion is
\[
\begin{cases}
p(x)\,I_0(x) + q(x)\,K_0(x) & k \text{ even},\\
p(x)\,I_1(x) - q(x)\,K_1(x) & k \text{ odd}.
\end{cases}
\]
We want both expressions to equal 1 identically in $x$. Solving the linear system, the determinant is $-\big(I_0(x)K_1(x) + I_1(x)K_0(x)\big) = -1/x$ by the Wronskian identity \eqref{eq-bessel-wronskian-general}, so the unique solution is
\[
p(x) = x\big[K_0(x) + K_1(x)\big], \qquad q(x) = x\big[I_1(x) - I_0(x)\big].
\]
Indeed, one checks directly that
\[
\begin{aligned}
p\,I_0 + q\,K_0 &= x\big[K_1 I_0 + I_1 K_0\big] = 1, \\
p\,I_1 - q\,K_1 &= x\big[K_0 I_1 + I_0 K_1\big] = 1.
\end{aligned}
\]
With this choice every carrier collapses to 1 and the coefficient of $t^n/n!$ becomes $\sum_k \mStirling{n}{k}_2 x^k$, as claimed. (All series converge absolutely for $x > 0$ and all real $t$, justifying the rearrangements.)
\end{proof}

\begin{corollary}[Row sums and parity row sums]\label{cor-2bell-rowsums}
Setting $x=1$ in Theorem \ref{thm-bb-polynomials} and comparing with Theorem \ref{thm-m2-egf}:
\[
\mathscr{B}_n(1) = \sum_{k=0}^{n}\mStirling{n}{k}_2 = B^{(2)}_n \quad (\textup{A007472}),
\]
since $G(t,1) = [K_0(1)+K_1(1)]I_0(e^t) + [I_1(1)-I_0(1)]K_0(e^t)$ is precisely the e.g.f.\ of A007472. Moreover, choosing instead the weights $p(x) = xK_1(x),\, q(x) = xI_1(x)$ (which map even-carrier coefficients to 1 and odd-carrier coefficients to 0) and $p(x) = xK_0(x),\, q(x) = -xI_0(x)$ (the reverse) shows in the same way that
\[
\sum_{\substack{k=0 \\ k \,\mathrm{even}}}^{n}\mStirling{n}{k}_2 = B^{(2,0)}_n \quad (\textup{A351143}), \qquad
\sum_{\substack{k=0 \\ k \,\mathrm{odd}}}^{n}\mStirling{n}{k}_2 = B^{(2,1)}_n \quad (\textup{A351028}).
\]
\end{corollary}

\begin{proof}
For the first display, Theorem \ref{thm-m2-egf} identifies $G(t,1)$ as the e.g.f.\ of A007472, while Theorem \ref{thm-bb-polynomials} expands the same function as $\sum_n \mathscr{B}_n(1)t^n/n!$. For the parity-restricted sums, the weights $p(x) = xK_1(x), q(x) = xI_1(x)$ solve $pI_0 + qK_0 = 1$, $pI_1 - qK_1 = 0$ (again via the Wronskian), so the resulting bivariate function expands with only the even-$k$ terms surviving; at $x = 1$ it equals $K_1(1)I_0(e^t) + I_1(1)K_0(e^t)$, the e.g.f.\ of A351143 by Theorem \ref{thm-m2-egf}. The odd case is identical.
\end{proof}

Corollary \ref{cor-2bell-rowsums} reveals a satisfying structure: \emph{the two primitive 2-Bell sequences are interleaved inside a single triangle}---A351143 collects the even-indexed columns and A351028 the odd-indexed ones, while their sum A007472 collects them all. In Section \ref{sec-general-case} (Corollary \ref{cor-mbell-rowsums}) we will prove by elementary means that this picture persists for every $m$, with the $m$ primitive $m$-Bell sequences appearing as the residue-class row sums of the $m$-Stirling triangle.

The Bessel--Stirling numbers are a novel generalization of the Stirling numbers and join a broader family of two-term recurrences \cite{mansour_general_2012,neuwirth_recursively_2001,barbero_g_bivariate_2014,spiveySolutionsGeneralCombinatorial2011}. We return to the connection between this new array and previous generalizations of the Stirling numbers in Section \ref{sec-final-remarks}. Note also that the Bessel--Stirling numbers are distinct from the $r$-Stirling numbers of Broder \cite{broder1984}, despite the superficially similar subscript notation.

Due to the form of their recurrence, the Bessel--Stirling numbers admit many identities common to other such two-term-recurrence triangular arrays (see Appendices A and B). The most important one is an explicit formula (proof given in Appendix A):

\begin{theorem}[Explicit formula for the Bessel--Stirling numbers of the second kind]\label{thm-bs-formula}
For all $n \geq 2$ and $k \geq 2$,
\[
    \mStirling{n}{k}_2 = \sum_{j=1}^{\floor{k/2}}(2j)^{n-k}\, c_{k,j} \;+\; (n-k+1)\sum_{j=1}^{\floor{(k-1)/2}}(2j)^{n-k}\, d_{k,j},
\]
where, for $1 \leq r \leq \floor{(k-1)/2}$,
\[
    d_{k,r} = \frac{(-r)^{k-1}}{r!\,r!\,(\floor{\frac{k}{2}}-r)!\,(\floor{\frac{k-1}{2}}-r)!},
\]
and, for $1 \leq r \leq \floor{k/2}$,
\[
    c_{k,r} = \begin{cases}
                - d_{k,r} \left(r(2H_{r-1}-H_{\floor{\frac{k}{2}}-r}-H_{\floor{\frac{k-1}{2}}-r})-k+3\right) & \text{if }\, 1 \leq r < k/2\\
                \frac{r^{2r}}{r!\,r!} & \text{if }\, r = k/2.
            \end{cases}
\]
For odd $k$ the two sums have the same range and may be merged; for even $k$ the endpoint $r = k/2$ carries a $c$-term but \emph{no} $d$-term---it corresponds to a simple rather than a double pole in the partial-fraction analysis of Appendix A---which is why the $d$-sum stops at $\floor{(k-1)/2}$. For $2 \leq n < k$ the right-hand side vanishes, as it must.

Alternatively, the following limit representation is valid for all $n\ge0, k\ge 0$:
\[
\begin{aligned}
    \mStirling{n}{k}_2 &= (-1)^{k-1}\, 2^{n-k} \sum_{j=0}^{\floor{k/2}} \lim_{t \to j} \left[t^{n-1}\,
        \frac
            {n-2-t\,\Psi_k(t)}
            {\Gamma(t+1)^{2}\, \Gamma\left(\floor{\tfrac{k}{2}}-t+1\right) \Gamma\left(\floor{\tfrac{k-1}{2}}-t+1\right)}\right],
        \\
        \Psi_k(t) &= 2\psi(t)-\psi\left(\floor{\tfrac{k}{2}}-t+1\right)-\psi\left(\floor{\tfrac{k-1}{2}}-t+1\right),
\end{aligned}
\]
where $H_n$ is the $n$-th harmonic number, $\psi(x)$ is the digamma function, and $\Gamma(x)$ is the gamma function.
\end{theorem}

The proof proceeds by partial-fraction decomposition of the ordinary generating function of the $k$-th column,
\begin{equation}\label{eq-column-ogf}
A_k(x) = \sum_{n\ge0} \mStirling{n}{k}_2\, x^n = \frac{x^k}{\prod_{j=1}^k\left(1-2\floor{j/2}x\right)},
\end{equation}
which itself follows directly from the recurrence \eqref{eq-bs-recurrence} (see Appendix A). Two useful byproducts of \eqref{eq-column-ogf} are worth recording here. First, explicit formulas for the low columns and the subdiagonal:
\[
    \begin{aligned}
        \mStirling{n}{2}_2 &= 2^{n-2}, & \quad n \ge 2\ \ (\text{rel. } \href{https://oeis.org/A000079}{A000079})\\
        \mStirling{n}{3}_2 &= (n-2)2^{n-3}, & \quad n \ge 3\ \ (\text{rel. } \href{https://oeis.org/A001787}{A001787})\\
        \mStirling{n}{4}_2 &= 2^{2n-6}-(n-1)\,2^{n-4}, & n\geq 4 \ \ (\text{rel. } \href{https://oeis.org/A100575}{A100575})\\
        \mStirling{n}{5}_2 &= n\,2^{n-5}+(n-6)\, 2^{2n-8}, & n\geq 5 \ \ (\text{rel. } \href{https://oeis.org/A158681}{A158681})\\
        \mStirling{n}{n-1}_2 &= \floor{(n-1)^2/2}, & n \ge 1 \ \ (\text{rel. } \href{https://oeis.org/A007590}{A007590})
    \end{aligned}
\]
(with the caveat that the sequences listed at the end of each line are offset or scaled). Second, expanding the product in \eqref{eq-column-ogf} as a geometric series shows that, for $n \geq k$, $\mStirling{n}{k}_2$ is the complete homogeneous symmetric polynomial $h_{n-k}$ evaluated at the weights $2\floor{1/2},$ $2\floor{2/2},$ $\ldots,$ $2\floor{k/2}$; we return to this symmetric-function viewpoint, and its dual for the first kind, in Section \ref{sec-first-kind} (Proposition \ref{prop-symmetric-functions}).

Now that we understand the analytic properties of the Bessel--Stirling and Bessel--Bell numbers and their basic recurrences, the natural next question is: what do these numbers count?

\section{Combinatorial interpretation of the Bessel--Stirling numbers of the second kind and the Bessel--Bell numbers}\label{sec-combinatorial}
\noindent The Bessel--Stirling numbers $\mStirling{n}{k}_2$ count the number of ways to place $n$ labeled marbles into $k$ unlabeled compartments following a parity-constrained urn model.
\begin{theorem}[The two-compartment urn model]\label{thm-urn-model}
The Bessel--Stirling number of the second kind $\mStirling{n}{k}_2$ counts the number of ways to arrange $n$ labeled marbles (placed one at a time in the order $1, 2, \ldots, n$) in $k$ total compartments distributed across urns with two unlabeled compartments each, subject to the following constraints:
\begin{itemize}
    \renewcommand\labelitemi{--}
    \item You can place a marble in any compartment of any urn to begin
    \item You cannot place another marble in an urn's used compartment until the other compartment of that urn also contains at least one marble
    \item You can freely place marbles in any compartment of urns where both compartments have been used
    \item You cannot open a new urn until both compartments of all previously used urns contain at least one marble each
\end{itemize}
Here both the urns and the two compartments within each urn are indistinguishable: two placements count as the same arrangement precisely when one is obtained from the other by permuting urns and by swapping the compartments within urns; unused compartments (including the empty partner compartment of a partially filled urn) are suppressed; and $k$ counts the non-empty compartments. Equivalently, one may order the urns, and the compartments within each urn, canonically by first use---the constraints then make ``open a new compartment'' a single well-defined move at every step. Consequently, the Bessel--Bell numbers (\href{https://oeis.org/A007472}{A007472}) count the total number of such parity-constrained placements of $n$ labeled marbles, over all $k$.
\end{theorem}

Before diving into the proof, let us go through some examples to build intuition. Let $\frac{a}{b}\big|\frac{c}{d}$ represent 2 urns with vertical bars separating the urns and a horizontal bar separating the two compartments within an urn. The numbers that will replace the placeholder letters represent the marbles, labeled in the order in which they were placed (one at a time).
For $\mStirling{3}{2}_2$ we have two ways to place the marbles in two compartments of a single urn: $\frac{1\,3}{2}$ or $\frac{1}{2\,3}$.
For $\mStirling{3}{3}_2$ we have one way to place each marble in its own compartment (2 urns, 3 compartments used): $\frac{1}{2}\big|\frac{3}{}$. All valid partitions of up to 5 objects are presented in Table \ref{tab-bs-enumeration}.
\bgroup
\renewcommand{\arraystretch}{2}
\begin{table}[htbp]
    \centering
        \begin{tabular}{|l|l|p{10.5cm}|}
        \hline
        T(n,k) & Count & Examples \\
        \hline
        T(1,1) & 1 & \texttt{$
            \frac{1}{}
        $}\\
        \hline
        T(2,2) & 1 & \texttt{$
            \frac{1}{2}
            \smallskip
        $}\\
        \hline
        T(3,2) & 2 & \texttt{$
            \frac{1}{2\,3},\
            \frac{1\,3}{2}
        $}\\
        T(3,3) & 1 & \texttt{$
            \frac{1}{2}\big|\frac{3}{}
            \smallskip
        $}\\
        \hline
        T(4,2) & 4 & \texttt{$
            \frac{1\,4}{2\,3},\
            \frac{1\,3\,4}{2},\
            \frac{1}{2\,3\,4},\
            \frac{1\,3}{2\,4}
        $} \\
        T(4,3) & 4 & \texttt{$
            \frac{1\,4}{2}\big|\frac{3}{},\
            \frac{1}{2\,4}\big|\frac{3}{},\
            \frac{1\,3}{2}\big|\frac{4}{},\
            \frac{1}{2\,3}\big|\frac{4}{}
        $}\\
        T(4,4) & 1 & \texttt{$
            \frac{1}{2}\big|\frac{3}{4}
            \smallskip
        $} \\
        \hline
        T(5,2) & 8 & \texttt{$
            \frac{1\,4\,5}{2\,3},\
            \frac{1\,4}{2\,3\,5},\
            \frac{1\,3\,4\,5}{2},\
            \frac{1\,3\,4}{2\,5},\
            \frac{1\,5}{2\,3\,4},\
            \frac{1}{2\,3\,4\,5},\
            \frac{1\,3\,5}{2\,4},\
            \frac{1\,3}{2\,4\,5}
        $}\\
        T(5,3) & 12 & \texttt{$
            \frac{1\,4\,5}{2}\big|\frac{3}{},\
            \frac{1\,4}{2\,5}\big|\frac{3}{\,},\
            \frac{1\,5}{2\,4}\big|\frac{3}{\,},\
            \frac{1}{2\,4\,5}\big|\frac{3}{\,},\
            \frac{1\,3\,5}{2}\big|\frac{4}{\,},\
            \frac{1\,3}{2\,5}\big|\frac{4}{\,},\
            \frac{1\,5}{2\,3}\big|\frac{4}{\,},\
            \frac{1}{2\,3\,5}\big|\frac{4}{\,},\
            \newline\newline
            \smallskip
            \frac{1\,3\,4}{2}\big|\frac{5}{\,},\
            \frac{1}{2\,3\,4}\big|\frac{5}{\,},\
            \frac{1\,3}{2\,4}\big|\frac{5}{\,},\
            \frac{1\,4}{2\,3}\big|\frac{5}{\,}
        $}\\

        T(5,4) & 8 & \texttt{$
            \frac{1\,3}{2} \big| \frac{4}{5},\
            \frac{1\,4}{2} \big| \frac{3}{5},\
            \frac{1\,5}{2} \big| \frac{3}{4},\
            \frac{1}{2\,3} \big| \frac{4}{5},\
            \frac{1}{2\,4} \big| \frac{3}{5},\
            \frac{1}{2\,5} \big| \frac{3}{4},\
            \frac{1}{2} \big| \frac{3\,5}{4},\
            \frac{1}{2} \big| \frac{3}{4\,5}
        $}\\
        T(5,5) & 1 & \texttt{$
            \frac{1}{2}\big|\frac{3}{4}\big|\frac{5}{\,}
            \smallskip
        $}\\
        \hline
        \end{tabular}
        \bigskip
    \caption{Valid Bessel--Stirling constrained partitions of $n$ objects into $k$ compartments of 2-compartment urns.}
    \label{tab-bs-enumeration}
\end{table}
\egroup
\begin{proof}
We proceed by induction on $n$. The base cases are straightforward:
\begin{itemize}
\item $\mStirling{0}{0}_2 = 1$: There is exactly one way to arrange 0 marbles using 0 compartments.
\item $\mStirling{n}{0}_2 = 0$ for $n > 0$: It is impossible to arrange a positive number of marbles using 0 compartments.
\item $\mStirling{n}{k}_2 = 0$ for $k > n$: It is impossible to use more than $n$ compartments when placing $n$ marbles, since each used compartment contains at least one marble.
\item $\mStirling{n}{n}_2 = 1$ for $n > 0$: There is exactly one way to place $n$ marbles into $n$ different compartments (the constraints force the order in which compartments open).
\end{itemize}
For the inductive step, assume that for some $n \geq 1$ and \emph{every} $k \geq 0$ simultaneously, the number $\mStirling{n}{k}_2$ correctly counts the arrangements of $n$ labeled marbles into $k$ compartments following our constraints (the step below uses the hypothesis at two adjacent values of $k$). We need to show that $\mStirling{n+1}{k}_2 = 2\lfloor k/2 \rfloor \mStirling{n}{k}_2 + \mStirling{n}{k-1}_2$ counts the arrangements for $n+1$ marbles in $k$ compartments.
Consider the placement options for the $(n+1)$-st marble:
\begin{itemize}
    \item \textit{Place in a new compartment}: The $(n+1)$-st marble can be placed in a new (previously unused) compartment, making the total number of compartments used equal to $k$. This requires that the previous $n$ marbles were arranged in $k-1$ compartments, and there are $\mStirling{n}{k-1}_2$ such arrangements. By the indistinguishability convention in the theorem statement, opening a new compartment is a \emph{single} option: the constraints determine whether the new compartment lies in the partially filled urn or in a fresh urn, and all fresh urns (and the two compartments of a fresh urn) are interchangeable.
    \item \textit{Place in a previously used compartment}: The $(n+1)$-st marble can be placed in one of the compartments already containing marbles. Here, the key insight comes from the urn model's constraints.
\end{itemize}

We need to count how many valid placement options exist among the $k$ compartments already in use. This depends on the distribution of these compartments across urns:
\begin{itemize}
    \item \textit{When $k$ is even}, say $k = 2j$, the $k$ compartments must be distributed across exactly $j$ fully filled urns (both compartments used in each). This is because our constraints prohibit opening a new urn until all previous urns have at least one marble in each compartment. With $j$ fully filled urns, we have $2j = k$ valid placement options for the $(n+1)$-st marble, since we can place it in any of the already-used compartments.
    \item \textit{When $k$ is odd}, say $k = 2j+1$, we must have $j$ fully filled urns (accounting for $2j$ compartments) plus one urn with exactly one compartment used. According to the constraints, we cannot place another marble in this partially-filled urn's used compartment until its other compartment also contains a marble. Therefore, we have only $2j = 2\lfloor k/2 \rfloor$ valid placement options in previously used compartments.
    \end{itemize}

In either case, the number of valid placement options in previously used compartments is $2\lfloor k/2 \rfloor$. Since there are $\mStirling{n}{k}_2$ ways to arrange the first $n$ marbles in $k$ compartments, we have $2\lfloor k/2 \rfloor \cdot \mStirling{n}{k}_2$ ways to place the $(n+1)$-st marble in a previously used compartment.
Combining both cases, the total number of ways to arrange $n+1$ marbles in $k$ compartments is
\[
\mStirling{n+1}{k}_2 = 2\lfloor k/2 \rfloor \cdot \mStirling{n}{k}_2 + \mStirling{n}{k-1}_2.
\]
This matches exactly the recurrence relation for the Bessel--Stirling numbers, completing the proof. (Two placement processes that produce the same final assignment of marbles to compartments coincide step by step, since the placement history can be recovered by removing marbles $n, n-1, \ldots$ in reverse order; the count is therefore genuinely a count of distinct arrangements.)
\end{proof}

This combinatorial interpretation generalizes verbatim to urns with $m$ compartments, yielding the $m$-Stirling numbers and $m$-Bell sequences (Theorem \ref{thm-general-urn}). In the special case $m=1$, each urn has only one compartment, so we are left with the standard interpretation of the Stirling numbers as counting partitions of labeled objects into unlabeled non-empty sets. The combinatorial model provides intuition for why the row sums shift two places to the left after applying the binomial transform twice: the constraint requiring both compartments of an urn to be filled before opening a new urn creates a ``pairing effect'', which manifests as the shift-by-two property. This loose argument is made precise, in a purely algebraic way, by the shift identity of Theorem \ref{thm-shift-identity}. Finally, an easy way to enumerate the combinatorial objects counted by the Bessel--Stirling numbers of the second kind is via succession rules and a generating tree, as introduced by West \cite{west_generating_1995}. See Figure \ref{fig:generating-tree} and Appendix C.

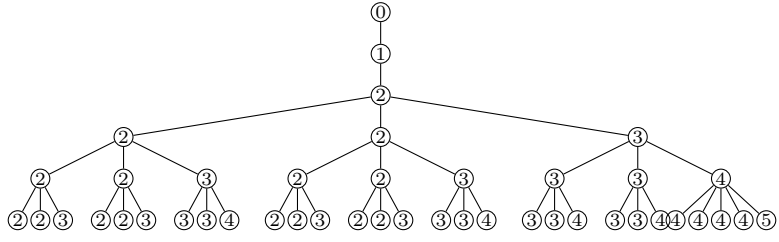
\begin{figure}[h]  
\centering
\begin{tikzpicture}[scale=0.5, level distance=1.1cm,
level 1/.style={sibling distance=8cm},
level 2/.style={sibling distance=7cm},
level 3/.style={sibling distance=6.8cm},
level 4/.style={sibling distance=2.2cm},
level 5/.style={sibling distance=0.6cm},
every node/.style={draw,circle,minimum size=0.25cm,inner sep=0pt,font=\scriptsize,}]
\node {0}
    child {node {1}
        child {node {2}
            child {node {2}
                child {node {2}
                    child {node {2}}
                    child {node {2}}
                    child {node {3}}
                }
                child {node {2}
                    child {node {2}}
                    child {node {2}}
                    child {node {3}}
                }
                child {node {3}
                    child {node {3}}
                    child {node {3}}
                    child {node {4}}
                }
            }
            child {node {2}
                child {node {2}
                    child {node {2}}
                    child {node {2}}
                    child {node {3}}
                }
                child {node {2}
                    child {node {2}}
                    child {node {2}}
                    child {node {3}}
                }
                child {node {3}
                    child {node {3}}
                    child {node {3}}
                    child {node {4}}
                }
            }
            child {node {3}
                child {node {3}
                    child {node {3}}
                    child {node {3}}
                    child {node {4}}
                }
                child {node {3}
                    child {node {3}}
                    child {node {3}}
                    child {node {4}}
                }
                child {node {4}
                    child {node {4}}
                    child {node {4}}
                    child {node {4}}
                    child {node {4}}
                    child {node {5}}
                }
            }
        }
    };
\end{tikzpicture}
\caption{The first six levels (levels $0$--$5$) of the generating tree for the Bessel--Stirling numbers of the second kind using the succession rule $\Omega$ of Appendix C. The nine level-4 nodes read $(2)^4(3)^4(4)$ and the twenty-nine level-5 nodes read $(2)^8(3)^{12}(4)^8(5)$.}
\label{fig:generating-tree}
\end{figure}

\section{Bessel--Stirling numbers of the first kind and Bessel falling factorials}\label{sec-first-kind}
\noindent The regular Stirling numbers come in a dual pair. In addition to the Stirling numbers of the second kind, which we explored at length, there exists a partner triangular array that bears the name Stirling numbers of the first kind. The two arrays are orthogonal to each other, or inverses of one another \cite{riordan_inverse_1964}:
\[
\sum_{j=k}^n s(n,j)S(j,k)=\sum_{j=k}^n(-1)^{n-j} \Stirlingone{n}{j}\Stirling{j}{k} = \delta_{n,k},
\]
where $s(n,k) = (-1)^{n-k}\Stirlingone{n}{k}$ denotes the \textit{signed} Stirling numbers of the first kind, $\Stirlingone{n}{k}$ the \textit{unsigned} Stirling numbers of the first kind, and $\delta_{n,k}$ is the Kronecker delta. When represented as lower-triangular matrices, with $S$ the matrix of Stirling numbers of the second kind and $s$ the matrix of signed Stirling numbers of the first kind, we have $sS = Ss = I$. The recurrence relation for the unsigned Stirling numbers of the first kind has a factor of $n$ rather than $k$:
\[
\Stirlingone{n+1}{k} = n\Stirlingone{n}{k} + \Stirlingone{n}{k-1}.
\]

The Bessel--Stirling numbers also have a first-kind partner, defined by the analogous recurrence.
\begin{definition}[Unsigned Bessel--Stirling numbers of the first kind]\label{def-bs1}
Let $\mStirlingone{n}{k}_2$ be defined for $n, k \geq 0$ by $\mStirlingone{0}{k}_2 = \delta_{0,k}$, the convention $\mStirlingone{n}{k}_2 = 0$ for $k<0$, and
    \[
    \mStirlingone{n+1}{k}_2 = 2\floor{n/2}\mStirlingone{n}{k}_2 + \mStirlingone{n}{k-1}_2.
    \]
\end{definition}

The first few rows of the unsigned Bessel--Stirling numbers of the first kind are shown in Table \ref{tab-bessel-stirling1}. This recurrence closely mirrors that of the standard Stirling numbers of the first kind, but with the factor $n$ replaced by $2\floor{n/2}$, reflecting the ``even-step'' structure inherent in the Bessel framework. Just as the ordinary Stirling numbers of the first kind arise in the context of a change of basis from ordinary powers to falling powers, these generalized coefficients appear as the change-of-basis matrix between ordinary powers and a novel class of falling factorials.

\begin{table}[htbp]
    \centering
    \begin{tabular}{|l|*{10}{c}|c|}
        \hline
          & k=0 & k=1 & k=2   & k=3   & k=4   & k=5   & k=6   & k=7   & k=8   & k=9   & $\sum_{k=0}^{n}$ \\
        \hline
        n=0 & 1   &     &       &       &       &       &       &       &       &       & 1                  \\
        n=1 & 0   & 1   &       &       &       &       &       &       &       &       & 1                  \\
        n=2 & 0   & 0   & 1     &       &       &       &       &       &       &       & 1                  \\
        n=3 & 0   & 0   & 2     & 1     &       &       &       &       &       &       & 3                  \\
        n=4 & 0   & 0   & 4     & 4     & 1     &       &       &       &       &       & 9                  \\
        n=5 & 0   & 0   & 16    & 20    & 8     & 1     &       &       &       &       & 45                 \\
        n=6 & 0   & 0   & 64    & 96    & 52    & 12    & 1     &       &       &       & 225                \\
        n=7 & 0   & 0   & 384   & 640   & 408   & 124   & 18    & 1     &       &       & 1575               \\
        n=8 & 0   & 0   & 2304  & 4224  & 3088  & 1152  & 232   & 24    & 1     &       & 11025              \\
        n=9 & 0   & 0   & 18432 & 36096 & 28928 & 12304 & 3008  & 424   & 32    & 1     & 99225              \\
        \hline
    \end{tabular}
    \bigskip
    \caption{Unsigned Bessel--Stirling numbers of the first kind, $\mStirlingone{n}{k}_2$. Row sums give \href{https://oeis.org/A000246}{A000246}.}
    \label{tab-bessel-stirling1}
\end{table}

\begin{table}[htbp]
    \centering
    \small
    \begin{tabular}{|l|*{10}{c}|c|}
        \hline
          & k=0 & k=1 & k=2   & k=3   & k=4   & k=5   & k=6   & k=7   & k=8   & k=9   & $\sum_{k=0}^{n}$ \\
        \hline
        n=0 & 1   &     &       &       &       &       &       &       &       &       & 1                  \\
        n=1 & 0   & 1   &       &       &       &       &       &       &       &       & 1                  \\
        n=2 & 0   & 0   & 1     &       &       &       &       &       &       &       & 1                  \\
        n=3 & 0   & 0   & 4     & 1     &       &       &       &       &       &       & 5                  \\
        n=4 & 0   & 0   & 16    & 8     & 1     &       &       &       &       &       & 25                 \\
        n=5 & 0   & 0   & 96    & 64    & 16    & 1     &       &       &       &       & 177                \\
        n=6 & 0   & 0   & 576   & 480   & 192   & 24    & 1     &       &       &       & 1273               \\
        n=7 & 0   & 0   & 4608  & 4416  & 2400  & 432   & 36    & 1     &       &       & 11893              \\
        n=8 & 0   & 0   & 36864 & 39936 & 28416 & 6720  & 864   & 48    & 1     &       & 112849             \\
        n=9 & 0   & 0   & 368640& 436224& 380928& 109056& 18816 & 1536  & 64    & 1     & 1315265            \\
        \hline
    \end{tabular}
    \bigskip
    \caption{Bessel--Lah numbers (Bessel--Stirling numbers of the third kind), $\mLah{n}{k}_2$.}
    \label{tab-bessel-stirling3}
\end{table}

\subsection{Bessel falling factorials}

The signed Stirling numbers of the first kind are well known to be the coefficients in the expansion of the falling factorial:
\[
x^{\underline{n}} =\prod_{k=0}^{n-1}(x - k)= x(x-1)(x-2)\cdots(x-n+1) = \sum_{k=0}^n (-1)^{n-k} \Stirlingone{n}{k} x^k.
\]

In the Bessel case, we discover a strikingly analogous structure.
\begin{definition}[Bessel falling and rising factorials]\label{def-bessel-factorials}
Define the Bessel falling and rising factorials (or even-step falling and rising factorials) as
\[
\begin{aligned}
    x^{\uuline{n}} &= \prod_{k=0}^{n-1}\left(x-2\floor{k/2}\right) = x\cdot x\cdot (x-2)(x-2)(x-4)(x-4)\cdots \\
    x^{\ooline{n}} &= \prod_{k=0}^{n-1}\left(x+2\floor{k/2}\right) = x\cdot x\cdot (x+2)(x+2)(x+4)(x+4)\cdots
\end{aligned}
\]
where each product has exactly $n$ factors, so each factor is repeated twice except possibly the last: e.g.\ $x^{\uuline{4}} = x^2(x-2)^2$ but $x^{\uuline{5}} = x^2(x-2)^2(x-4)$.
\end{definition}
\begin{theorem}\label{thm-bessel-factorial-expansion}
The Bessel--Stirling numbers of the first kind are the coefficients in the expansions of the Bessel falling and rising factorials:
\[
\begin{aligned}
    x^{\uuline{n}} &= \sum_{k=0}^n (-1)^{n-k} \mStirlingone{n}{k}_2\, x^k\\
    x^{\ooline{n}} &= \sum_{k=0}^n \mStirlingone{n}{k}_2\, x^k.
\end{aligned}
\]
\end{theorem}
\begin{proof}
Induction on $n$, rising version first. Multiplying $x^{\ooline{n}} = \sum_k \mStirlingone{n}{k}_2 x^k$ by the next factor $(x + 2\floor{n/2})$ sends the coefficient of $x^k$ to $\mStirlingone{n}{k-1}_2 + 2\floor{n/2}\mStirlingone{n}{k}_2$, which is exactly the recurrence of Definition \ref{def-bs1}. The falling version follows by substituting $x \mapsto -x$, since $x^{\uuline{n}} = (-1)^n\, (-x)^{\ooline{n}}$.
\end{proof}
Here are the expansions of the first few Bessel rising factorials:
\[
\begin{aligned}
     x^{\ooline{1}} = x &= x \\
     x^{\ooline{2}} = x^2 &= x^2 \\
     x^{\ooline{3}} = x^2(x+2) &= 2x^2+x^3 \\
     x^{\ooline{4}} = x^2(x+2)^2 &= 4x^2+4x^3+x^4 \\
     x^{\ooline{5}} = x^2(x+2)^2(x+4) &= 16x^2+20x^3+8x^4+x^5 \\
     x^{\ooline{6}} = x^2(x+2)^2(x+4)^2 &= 64x^2+96x^3+52x^4+12x^5 +x^6.
\end{aligned}
\]
Just as the regular falling powers form an alternative polynomial basis, which can be converted to and from the monomial basis using the Stirling numbers, the same is true for the Bessel falling powers. One recovers a monomial from them using the Bessel--Stirling numbers of the second kind:

\begin{theorem}\label{thm-inverse-expansion}
\[
x^n =\sum_{k=0}^n \mStirling{n}{k}_2\; x^{\uuline{k}}.
\]
\end{theorem}
\begin{proof}
From Definition \ref{def-bessel-factorials}, $x^{\uuline{k+1}} = x^{\uuline{k}}(x - 2\floor{k/2})$, so
\begin{equation}\label{eq-x-action-falling}
x \cdot x^{\uuline{k}} = x^{\uuline{k+1}} + 2\floor{k/2}\, x^{\uuline{k}}.
\end{equation}
Now induct on $n$: multiplying $x^n = \sum_k \mStirling{n}{k}_2 x^{\uuline{k}}$ by $x$ and applying \eqref{eq-x-action-falling} term by term shows that the coefficient of $x^{\uuline{k}}$ in $x^{n+1}$ is $\mStirling{n}{k-1}_2 + 2\floor{k/2}\mStirling{n}{k}_2 = \mStirling{n+1}{k}_2$, by \eqref{eq-bs-recurrence}. Since $\{x^{\uuline{k}}\}_{k \ge 0}$ has one polynomial of each degree, it is a basis and the coefficients are unique.
\end{proof}

Comparing \eqref{eq-x-action-falling} with the carrier computation in the proof of Theorem \ref{thm-carrier-expansion} explains the coincidence of coefficients: multiplication by $x$ acts on the falling-factorial basis exactly as $xD$ acts on the carrier module $\{x^k f_{k \bmod 2}\}$---both are ``weighted shift'' operators with weight sequence $2\floor{k/2}$.

\begin{proposition}[Orthogonality]\label{prop-orthogonality}
    \[
        \sum_{j=k}^n (-1)^{n-j}\mStirlingone{n}{j}_2\mStirling{j}{k}_2=\delta_{n,k},
        \qquad
        \sum_{j=k}^n (-1)^{j-k}\mStirling{n}{j}_2\mStirlingone{j}{k}_2=\delta_{n,k}.
    \]
\end{proposition}
\begin{proof}
By Theorem \ref{thm-bessel-factorial-expansion}, the matrix $\big((-1)^{n-k}\mStirlingone{n}{k}_2\big)_{n,k}$ converts the monomial basis into the Bessel-falling basis, and by Theorem \ref{thm-inverse-expansion} the matrix $\big(\mStirling{n}{k}_2\big)_{n,k}$ converts it back. The composition of the two changes of basis, in either order, is the identity.
\end{proof}

Thus the two Bessel--Stirling triangles form a pair of lower-triangular inverse matrices (after attaching signs to one of them)---one converting monomials to generalized factorials, the other reversing the transformation, in complete parallel with the classical Stirling duality.

\subsection{Symmetric function expressions and file placements}

Both triangles are specializations of the classical bases of symmetric functions, evaluated on the ``doubled'' weight sequence.

\begin{proposition}\label{prop-symmetric-functions}
Let $w_j = 2\floor{j/2}$, and let $e_d$ and $h_d$ denote the elementary and complete homogeneous symmetric polynomials of degree $d$. Then for $0 \le k \le n$:
\[
\mStirlingone{n}{k}_2 = e_{n-k}\big(w_0, w_1, \ldots, w_{n-1}\big), \qquad
\mStirling{n}{k}_2 = h_{n-k}\big(w_1, w_2, \ldots, w_{k}\big).
\]
\end{proposition}
\begin{proof}
The first identity reads off the coefficient of $x^k$ in $x^{\ooline{n}} = \prod_{j=0}^{n-1}(x + w_j)$ (Theorem \ref{thm-bessel-factorial-expansion}). The second reads off the coefficient of $x^{n}$ in the column generating function \eqref{eq-column-ogf}: expanding each geometric factor $\left(1-w_jx\right)^{-1} = \sum_i w_j^i x^i$ and extracting $[x^{n-k}]$ gives precisely the sum of all monomials $w_{j_1}w_{j_2}\cdots w_{j_{n-k}}$ with $1 \le j_1 \le \cdots \le j_{n-k} \le k$, which is $h_{n-k}(w_1, \ldots, w_k)$.
\end{proof}

The classical Stirling numbers are the same specializations with weights $(0, 1, 2, \ldots)$: $\Stirlingone{n}{k} = e_{n-k}(0,1,\ldots,n-1)$ and $\Stirling{n}{k} = h_{n-k}(1, 2, \ldots, k)$ \cite{comtet1974}. In rook-theoretic language, $e_{n-k}(w_0,\ldots,w_{n-1})$ counts \emph{file placements} of $n-k$ rooks on the ``doubled staircase'' board whose $j$-th column has height $w_j = 2\floor{j/2}$ (rooks in distinct columns, one cell each, with no constraint along rows), giving the Bessel--Stirling numbers of the first kind a direct board-counting interpretation. Placements of this kind fit the $q$-rook framework of Garsia and Remmel \cite{garsiaRemmel1986}; for the classical theory of non-attacking rook placements on Ferrers boards see Goldman, Joichi and White \cite{goldman_rook_1975}.

\subsection{A combinatorial interpretation of the first kind: parity-restricted insertions}

The unsigned Stirling numbers of the first kind $\Stirlingone{n}{k}$ count permutations of $[n] = \{1, \ldots, n\}$ with exactly $k$ cycles. The standard proof builds a permutation by inserting the elements $1, 2, \ldots, n$ one at a time: element $i+1$ either opens a new cycle, or is inserted immediately after any of the $i$ existing elements in its cycle. The Bessel--Stirling numbers of the first kind count the permutations produced when the insertion process is parity-restricted.

\begin{theorem}[Restricted-insertion model]\label{thm-insertion-model}
Build permutations of $[n]$ by inserting elements $1, 2, \ldots, n$ in order, where element $i+1$ either
\begin{itemize}
    \renewcommand\labelitemi{--}
    \item opens a new cycle as a fixed point, or
    \item is inserted immediately after one of the existing elements $1, \ldots, i$ in cycle notation, \emph{except} that when $i$ is odd, insertion immediately after the largest element $i$ is forbidden.
\end{itemize}
Distinct insertion histories produce distinct permutations, and $\mStirlingone{n}{k}_2$ counts the resulting permutations with exactly $k$ cycles. Equivalently, $\mStirlingone{n}{k}_2$ is the number of permutations $\pi$ of $[n]$ with $k$ cycles such that for every even $e \leq n$, the restriction of $\pi$ to $\{1, \ldots, e\}$ (obtained by deleting the elements larger than $e$ from the cycle notation) does not map $e - 1$ to $e$.
\end{theorem}

\begin{proof}
At the step inserting element $i+1$ there is $1$ way to open a new cycle and $i - (i \bmod 2) = 2\floor{i/2}$ permitted insertion positions, so the number of histories with $k$ ``new cycle'' steps satisfies the recurrence of Definition \ref{def-bs1} with the cycle count increasing exactly on new-cycle steps. Histories biject with outcomes because the history is recoverable from the final permutation: deleting the largest element (splicing it out of its cycle) undoes the last step, and iterating recovers every step. Finally, the insertion of $i+1$ immediately after $i$ is precisely what makes the restriction of the final permutation to $\{1,\ldots,i+1\}$ map $i$ to $i+1$; with $e = i + 1$ ranging over even values, the forbidden histories correspond exactly to the stated condition on $\pi$.
\end{proof}

For general $m$ the same model applies with ``$i$ odd'' replaced by the rule that insertion is only permitted after the $m\floor{i/m}$ smallest elements (equivalently, insertion after the largest $i \bmod m$ elements is forbidden); the count of reachable permutations of $[n]$ with $k$ cycles is then $\mStirlingone{n}{k}_m$ (Section \ref{sec-general-case}).

\subsection{Row sums of the first kind: permutations of odd order}

Although the first-kind triangle was, prior to this work, also not present in the OEIS, its row sums yield a sequence with a rich combinatorial history: \href{https://oeis.org/A000246}{A000246}, the number of permutations in the symmetric group $S_n$ that have odd order, i.e.\ all of whose cycle lengths are odd: $1, 1, 1, 3, 9, 45, 225, 1575, 11025, 99225, \ldots$

\begin{proposition}\label{prop-a000246}
\[
\sum_{k=0}^{n}\mStirlingone{n}{k}_2 = \prod_{j=0}^{n-1}\left(2\floor{j/2}+1\right) = \#\{\pi \in S_n : \pi \text{ has odd order}\}.
\]
\end{proposition}
\begin{proof}
Summing the recurrence of Definition \ref{def-bs1} over $k$ shows that the row sums $s_n$ satisfy $s_{n+1} = (2\floor{n/2}+1)s_n$, $s_0 = 1$, which telescopes to the product. For the second equality, let $a_n$ denote the number of permutations of $[n]$ all of whose cycles have odd length. By the exponential formula, $\sum_n a_n x^n/n! = \exp\big(\sum_{d \text{ odd}} x^d/d\big) = \exp\big(\tfrac{1}{2}\log\tfrac{1+x}{1-x}\big) = \sqrt{(1+x)/(1-x)} =: A(x)$. Differentiating, $(1-x^2)A'(x) = A(x)$, so extracting coefficients gives $a_{n+1} = a_n + n(n-1)a_{n-1}$ for $n \geq 1$, with $a_0 = a_1 = 1$. The product sequence $s_n$ satisfies the same recurrence for $n \geq 1$: $s_{n+1} - s_n = 2\floor{n/2}\,s_n = 2\floor{n/2}\big(2\floor{(n-1)/2}+1\big)s_{n-1} = n(n-1)s_{n-1}$, where the last equality holds for $n$ even ($n \cdot (n-1)$) and $n$ odd ($(n-1)\cdot n$) alike. Since $s_0 = a_0 = 1$ and $s_1 = a_1 = 1$, the sequences coincide.
\end{proof}

We emphasize that while the row sums agree, the triangle $\mStirlingone{n}{k}_2$ does \emph{not} refine the odd-order permutations by their number of cycles: for $n = 5$ the odd-order permutations count $(24, 20, 1)$ by $(1, 3, 5)$ cycles, whereas row 5 of the triangle reads $(16, 20, 8, 1)$ for $k = (2,3,4,5)$. Finding a \emph{natural} statistic on odd-order permutations that is equidistributed with the $k$-statistic of Theorem \ref{thm-insertion-model}---equivalently, an \emph{explicit} bijection between odd-order permutations and the parity-admissible permutations of Theorem \ref{thm-insertion-model} (any two equinumerous finite sets admit some bijection, so it is the explicitness that carries the content)---remains an open problem (Section \ref{sec-final-remarks}).

\subsection{Bessel--Lah numbers (the third kind)}

The classical Lah numbers $\Lah{n}{k} = \binom{n-1}{k-1}\frac{n!}{k!}$ convert rising factorials into falling factorials, and equal the matrix product of the unsigned first-kind and second-kind Stirling triangles. The same construction goes through in the Bessel setting.

\begin{definition}[Bessel--Lah numbers]
Define $\mLah{n}{k}_2$ as the connection coefficients between the Bessel rising and falling factorial bases:
\[
x^{\ooline{n}} = \sum_{k=0}^n \mLah{n}{k}_2\; x^{\uuline{k}}.
\]
\end{definition}

\begin{theorem}\label{thm-bessel-lah}
The Bessel--Lah numbers satisfy:
\begin{enumerate}
    \item[(i)] $\mLah{n}{k}_2 = \sum_{j} \mStirlingone{n}{j}_2 \mStirling{j}{k}_2$ \,(matrix product of the two Bessel--Stirling triangles);
    \item[(ii)] the two-term recurrence
    \[
    \mLah{n+1}{k}_2 = \left(2\floor{n/2} + 2\floor{k/2}\right)\mLah{n}{k}_2 + \mLah{n}{k-1}_2,
    \]
    with $\mLah{0}{k}_2 = \delta_{0,k}$.
\end{enumerate}
\end{theorem}
\begin{proof}
(i) Compose the expansions of Theorems \ref{thm-bessel-factorial-expansion} and \ref{thm-inverse-expansion}: $x^{\ooline{n}} = \sum_j \mStirlingone{n}{j}_2 x^j = \sum_j \mStirlingone{n}{j}_2 \sum_k \mStirling{j}{k}_2 x^{\uuline{k}}$.
(ii) Multiply the defining expansion by the next rising factor and use \eqref{eq-x-action-falling}:
\[
x^{\ooline{n+1}} = \left(x + 2\floor{n/2}\right)\sum_k \mLah{n}{k}_2 x^{\uuline{k}}
= \sum_k \mLah{n}{k}_2 \left( x^{\uuline{k+1}} + 2\floor{k/2}x^{\uuline{k}} + 2\floor{n/2}x^{\uuline{k}}\right),
\]
and collect the coefficient of $x^{\uuline{k}}$, using uniqueness of coefficients in the basis $\{x^{\uuline{k}}\}$.
\end{proof}

The first few rows are shown in Table \ref{tab-bessel-stirling3}. Note how the recurrence interpolates the classical Lah recurrence $\Lah{n+1}{k} = (n+k)\Lah{n}{k} + \Lah{n}{k-1}$, doubling-and-flooring both parameters. The row sums $1, 1, 1, 5, 25, 177, 1273, 11893, 112849, 1315265, \ldots$ (the Bessel analogue of A000262, the number of sets of lists) do not currently appear in the OEIS.

Figure \ref{fig:commutative-diagram} summarizes how all six triangles act as conversion operators among the five polynomial bases.

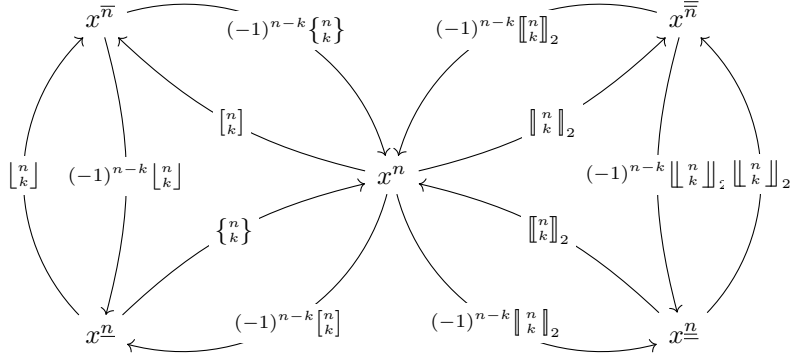
\begin{figure}[h]
\centering
\[
\begin{tikzcd}[column sep=9em, row sep=4.5em]
    x^{\overline{n}}
        \arrow[rd, bend left=45, "(-1)^{n-k}\Stirling{n}{k}" description]
        \arrow[dd, bend left=15, "{(-1)^{n-k}\Lah{n}{k}}" description]
    & &
    x^{\ooline{n}}
        \arrow[ld, bend left=-45, swap, "(-1)^{n-k}\mStirling{n}{k}_2" description]
        \arrow[dd, bend left=-15, swap, "{(-1)^{n-k}\mLah{n}{k}_2}" description]
    \\
    &
    x^{n}
        \arrow[ul, bend left=15, "\Stirlingone{n}{k}" description]
        \arrow[ur, bend left=-15, swap, "\mStirlingone{n}{k}_2" description]
        \arrow[dl, bend left=45, "(-1)^{n-k}\Stirlingone{n}{k}" description]
        \arrow[dr, bend left=-45, swap, "(-1)^{n-k}\mStirlingone{n}{k}_2" description]
    &\\
    x^{\underline{n}}
        \arrow[uu, bend left=45, "{\Lah{n}{k}}" description]
        \arrow[ur, bend left=15, "\Stirling{n}{k}" description]
    & &
    x^{\uuline{n}}
        \arrow[uu, bend left=-45, swap, "{\mLah{n}{k}_2}" description]
        \arrow[ul, bend left=-15, swap, "\mStirling{n}{k}_2" description]
\end{tikzcd}
\]
\caption{Stirling numbers and Bessel--Stirling numbers as transformations between polynomial bases. The regular Stirling (and Lah) numbers convert between ordinary powers (middle), falling powers (bottom left), and rising powers (top left). The Bessel--Stirling numbers convert between ordinary powers, Bessel falling powers (bottom right), and Bessel rising powers (top right). An arrow from basis $u$ to basis $v$ labeled $c_{n,k}$ encodes the expansion $v_n = \sum_k c_{n,k}\,u_k$.}
\label{fig:commutative-diagram}
\end{figure}

\section{The general case}\label{sec-general-case}
\noindent Having explored the case $m=2$ in depth, we now develop the theory for arbitrary $m \geq 1$. Remarkably, the central structural result---that the $m$-Stirling triangle's row sums produce the $m$-Bell numbers, with the residue classes of $k \bmod m$ producing the primitive sequences---admits a completely elementary proof, independent of any special function theory. We give that proof first, and then develop the analytic counterparts (hyper-Bessel carrier functions, hypergeometric e.g.f.s, and a general Dobi\'nski formula).

\subsection{The $m$-Stirling triangle and the shift identity}

\begin{definition}[$m$-Stirling numbers of the second kind]
For an integer $m \geq 1$, define $\mStirling{n}{k}_m$ for $n, k \geq 0$ by $\mStirling{0}{k}_m = \delta_{0,k}$, the convention $\mStirling{n}{k}_m = 0$ for $k < 0$, and the recurrence
\begin{equation}\label{eq-mstirling-recurrence}
\mStirling{n+1}{k}_m = m\floor{k/m} \, \mStirling{n}{k}_m + \mStirling{n}{k-1}_m.
\end{equation}
For $m = 1$ this is the classical Stirling triangle of the second kind; for $m=2$, the Bessel--Stirling triangle of Section \ref{sec-2-stirling}. Table \ref{tab-3stirling} displays the triangle for $m = 3$.
\end{definition}

\begin{table}[htbp]
    \centering
        \begin{tabular}{|l|*{10}{c}|c|}
        \hline
              & k=0 & k=1 & k=2 & k=3 & k=4 & k=5 & k=6 & k=7 & k=8 & k=9 & $\sum_{k=0}^{n}$ \\
        \hline
        n=0 & 1   &     &     &     &     &     &     &     &     &     & 1                \\
        n=1 & 0   & 1   &     &     &     &     &     &     &     &     & 1                \\
        n=2 & 0   & 0   & 1   &     &     &     &     &     &     &     & 1                \\
        n=3 & 0   & 0   & 0   & 1   &     &     &     &     &     &     & 1                \\
        n=4 & 0   & 0   & 0   & 3   & 1   &     &     &     &     &     & 4                \\
        n=5 & 0   & 0   & 0   & 9   & 6   & 1   &     &     &     &     & 16               \\
        n=6 & 0   & 0   & 0   & 27  & 27  & 9   & 1   &     &     &     & 64               \\
        n=7 & 0   & 0   & 0   & 81  & 108 & 54  & 15  & 1   &     &     & 259              \\
        n=8 & 0   & 0   & 0   & 243 & 405 & 270 & 144 & 21  & 1   &     & 1084             \\
        n=9 & 0   & 0   & 0   & 729 & 1458& 1215& 1134& 270 & 27  & 1   & 4834             \\
        \hline
        \end{tabular}
        \bigskip
    \caption{The $3$-Stirling numbers of the second kind, $\mStirling{n}{k}_3$.}
    \label{tab-3stirling}
\end{table}

\begin{lemma}\label{lem-identity-rows}
$\mStirling{n}{k}_m = 0$ for $k > n$, $\mStirling{n}{n}_m = 1$, and the first $m+1$ rows form an identity block:
\[
\mStirling{n}{k}_m = \delta_{n,k} \qquad \text{for } 0 \leq n \leq m.
\]
\end{lemma}
\begin{proof}
The first two claims follow by induction as in the classical case. For the third, note that the coefficient $m\floor{k/m}$ in \eqref{eq-mstirling-recurrence} vanishes whenever $k < m$. Hence, for rows $n < m$, where only entries with $k \leq n < m$ can be nonzero, the recurrence degenerates to a pure shift $\mStirling{n+1}{k}_m = \mStirling{n}{k-1}_m$, which propagates the initial row $\delta_{0,k}$ diagonally. For row $n = m$, the only new entry is $\mStirling{m}{m}_m = m\floor{m/m}\mStirling{m-1}{m}_m + \mStirling{m-1}{m-1}_m = 0 + 1 = 1$.
\end{proof}

We now come to the central result of the paper, an $m$-fold analogue of the classical shift identity \eqref{eq-stirling-shift}. It explains, at the level of individual triangle entries, why the row sums shift by $m$ places after $m$ binomial transforms.

\begin{theorem}[Shift identity for $m$-Stirling numbers]\label{thm-shift-identity}
For all integers $n \geq 0$ and $k \geq 0$,
\begin{equation}\label{eq-shift-identity}
\mStirling{n+m}{k}_m = \sum_{j=0}^{n}\binom{n}{j}\, m^{n-j}\, \mStirling{j}{k-m}_m.
\end{equation}
For $m = 1$ this is the classical identity \eqref{eq-stirling-shift}.
\end{theorem}

\begin{proof}
Induction on $n$, with $k$ arbitrary. For $n = 0$ the claim reads $\mStirling{m}{k}_m = \mStirling{0}{k-m}_m = \delta_{m,k}$, which is Lemma \ref{lem-identity-rows}. Assume \eqref{eq-shift-identity} holds for some $n \geq 0$ and all $k$. Using Pascal's rule $\binom{n+1}{j} = \binom{n}{j} + \binom{n}{j-1}$ and re-indexing the second sum,
\[
\sum_{j=0}^{n+1}\binom{n+1}{j}m^{n+1-j}\mStirling{j}{k-m}_m
= m\sum_{j=0}^{n}\binom{n}{j}m^{n-j}\mStirling{j}{k-m}_m + \sum_{j=0}^{n}\binom{n}{j}m^{n-j}\mStirling{j+1}{k-m}_m.
\]
Expanding $\mStirling{j+1}{k-m}_m$ by the recurrence \eqref{eq-mstirling-recurrence}, and noting $\floor{(k-m)/m} = \floor{k/m} - 1$, gives
\[
\mStirling{j+1}{k-m}_m = \left(m\floor{k/m} - m\right)\mStirling{j}{k-m}_m + \mStirling{j}{k-m-1}_m.
\]
Substituting this into the previous display, the two terms proportional to $m\,\mStirling{j}{k-m}_m$ cancel, leaving
\[
\begin{aligned}
\sum_{j=0}^{n+1}\binom{n+1}{j}m^{n+1-j}\mStirling{j}{k-m}_m
={}& m\floor{k/m} \sum_{j=0}^{n}\binom{n}{j}m^{n-j}\mStirling{j}{k-m}_m \\
&+ \sum_{j=0}^{n}\binom{n}{j}m^{n-j}\mStirling{j}{k-1-m}_m.
\end{aligned}
\]
By the induction hypothesis (applied at $k$ and at $k-1$), the right side equals $m\floor{k/m}\mStirling{n+m}{k}_m$ ${}+ \mStirling{n+m}{k-1}_m$, which is $\mStirling{n+m+1}{k}_m$ by \eqref{eq-mstirling-recurrence}, completing the induction. (All manipulations remain valid at the boundaries: we extend the array by $\mStirling{n}{k}_m = 0$ for $k < 0$, and with this extension the recurrence \eqref{eq-mstirling-recurrence} holds for every integer $k$---for $k < 0$ both sides vanish---so applying it in the column $k-m$, which is negative when $k < m$, is legitimate; the floor identity $\floor{(k-m)/m} = \floor{k/m} - 1$ also holds for all integers $k$.)
\end{proof}

\begin{corollary}[$m$-Bell numbers as row sums; primitive sequences as residue-class sums]\label{cor-mbell-rowsums}
Fix $m \geq 1$ and let
\[
\rho_r(n) = \sum_{\substack{k \geq 0 \\ k \equiv r \ (\mathrm{mod}\ m)}} \mStirling{n}{k}_m \quad (0 \leq r < m),
\qquad
\sigma(n) = \sum_{k=0}^{n} \mStirling{n}{k}_m .
\]
Then:
\begin{enumerate}
    \item[(i)] each $\rho_r$ satisfies the $m$-Bell recurrence \eqref{eq:m-shift-recurrence} with initial conditions $\rho_r(n) = \delta_{n,r}$ for $0 \leq n < m$; hence $\rho_r = B^{(m,r)}$ is the $r$-th primitive $m$-Bell sequence;
    \item[(ii)] the full row sums satisfy $\sigma = \sum_{r=0}^{m-1}\rho_r = B^{(m)}$, the composite $m$-Bell sequence;
    \item[(iii)] every sequence satisfying \eqref{eq:m-shift-recurrence} is the unique linear combination $\sum_{r} a_r B^{(m,r)}$ of the residue-class row sums determined by its initial values $a_0, \ldots, a_{m-1}$.
\end{enumerate}
\end{corollary}
\begin{proof}
(i) Summing the shift identity \eqref{eq-shift-identity} over all $k \equiv r \pmod m$, and noting that $k \mapsto k - m$ preserves the residue class, gives
\[
\rho_r(n+m) = \sum_{j=0}^{n}\binom{n}{j}m^{n-j}\rho_r(j),
\]
which is exactly \eqref{eq:m-shift-recurrence}. The initial values follow from Lemma \ref{lem-identity-rows}: for $n < m$, row $n$ contains the single entry $1$ at $k = n$, contributing to $\rho_r$ if and only if $n \equiv r$, i.e.\ $n = r$. (ii) Sum (i) over $r$. (iii) A solution of \eqref{eq:m-shift-recurrence} is determined by its first $m$ values, and the primitives form a basis for the space of initial conditions.
\end{proof}

For $m = 2$, Corollary \ref{cor-mbell-rowsums} recovers Corollary \ref{cor-2bell-rowsums} by entirely elementary means. Table \ref{tab-mbell-numbers} lists the composite $m$-Bell numbers for $m \le 4$ and the primitive sequences for $m = 3$.

\begin{table}[htbp]
    \centering
    \begin{tabular}{|l|l|}
    \hline
    Sequence & First terms \\
    \hline
    $B^{(1)}$ (Bell, A000110) & 1, 1, 2, 5, 15, 52, 203, 877, 4140, 21147, 115975, 678570 \\
    $B^{(2)}$ (A007472) & 1, 1, 1, 3, 9, 29, 105, 431, 1969, 9785, 52145, 296155 \\
    $B^{(3)}$ & 1, 1, 1, 1, 4, 16, 64, 259, 1084, 4834, 23635, 128377 \\
    $B^{(4)}$ & 1, 1, 1, 1, 1, 5, 25, 125, 625, 3129, 15745, 80265 \\
    \hline
    $B^{(3,0)}$ & 1, 0, 0, 1, 3, 9, 28, 96, 387, 1864, 10242, 60534 \\
    $B^{(3,1)}$ & 0, 1, 0, 0, 1, 6, 27, 109, 426, 1728, 7858, 42084 \\
    $B^{(3,2)}$ & 0, 0, 1, 0, 0, 1, 9, 54, 271, 1242, 5535, 25759 \\
    \hline
    \end{tabular}
    \bigskip
    \caption{Composite $m$-Bell numbers for $m = 1, \ldots, 4$, and the three primitive $3$-Bell sequences (residue-class row sums of Table \ref{tab-3stirling}). The sequences for $m \ge 3$ are not currently in the OEIS.}
    \label{tab-mbell-numbers}
\end{table}

\subsection{The ODE in hypergeometric form and hyper-Bessel carriers}

We next generalize the analytic theory of Sections \ref{sec-m2} and \ref{sec-2-stirling}. As in the case $m=2$, set $t=e^x$ (so that $D_x = tD_t$) and $f(t)=\mathcal{A}(x)$, and recall the classical operator identity \cite{comtet1974, dattoliTouchardPolynomialsGeneralized2010}
\begin{equation}\label{eq-xd-operator-identity}
    (tD_t)^n = \sum_{k=0}^{n}\Stirling{n}{k}t^kD_t^k,
\end{equation}
which is itself yet another manifestation of the Stirling numbers as change-of-basis coefficients (from the basis $(tD)^k$ to the basis $t^kD^k$ of the Weyl algebra).

\begin{proposition}\label{prop-general-ode}
The e.g.f.\ $\mathcal{A}(x)$ of a sequence satisfying \eqref{eq:m-shift-recurrence} equals $f(e^x)$, where $f$ satisfies the $m$-th order ODE
\begin{equation}\label{eq-general-ode-t}
(tD_t)^m f = t^m f, \qquad \text{equivalently} \qquad \sum_{k=0}^m\Stirling{m}{k}\,t^k f^{(k)}(t) = t^m f(t).
\end{equation}
For $m=1$ this is $tf' = tf$, solved by $e^t$; for $m=2$ it is the modified Bessel equation \eqref{eq-m2-ode}.
\end{proposition}
\begin{proof}
By Proposition \ref{prop:mbell-egf}, $D_x^m \mathcal{A} = e^{mx}\mathcal{A}$. Substituting $D_x = tD_t$ and $e^{mx} = t^m$ gives $(tD_t)^m f = t^m f$; expanding via \eqref{eq-xd-operator-identity} gives the second form.
\end{proof}

Note that \eqref{eq-general-ode-t} is precisely the primary-carrier equation \eqref{eq-carrier-ode}. The distinguished solution analytic at $t = 0$ is a generalized hypergeometric function, and the full carrier system consists of the \emph{hyper-Bessel functions} of Delerue \cite{delerue1953, kiryakova1994}:

\begin{proposition}[Hyper-Bessel carrier systems]\label{prop-hyper-bessel}
For $m \geq 1$ define
\[
\begin{aligned}
C^{(m)}_0(x) &= \sum_{k=0}^\infty \frac{(x/m)^{mk}}{(k!)^m} = {}_0F_{m-1}\!\left(;1,\ldots,1;\left(\frac{x}{m}\right)^{m}\right), \\
C^{(m)}_r(x) &= \sum_{k=0}^\infty \frac{(x/m)^{mk+m-r}}{(k!)^{r}\,((k+1)!)^{m-r}} = \frac{(x/m)^{m-r}}{(1!)^{m-r}}\; {}_0F_{m-1}\!\left(;\underbrace{1,\ldots,1}_{r-1},\underbrace{2,\ldots,2}_{m-r};\left(\frac{x}{m}\right)^{m}\right)
\end{aligned}
\]
for $1 \le r \le m-1$. Then $(C^{(m)}_0, \ldots, C^{(m)}_{m-1})$ is a carrier system of modulus $m$ in the sense of Definition \ref{def-carrier-system}:
\[
(xD + r)\,C^{(m)}_r = x\, C^{(m)}_{(r+1) \bmod m}, \qquad r = 0, \ldots, m-1.
\]
In particular $C^{(m)}_0$ satisfies $(xD)^m C^{(m)}_0 = x^m C^{(m)}_0$. For $m=1$, $C^{(1)}_0 = e^x$; for $m = 2$, $(C^{(2)}_0, C^{(2)}_1) = (I_0, I_1)$.
\end{proposition}
\begin{proof}
Direct computation on the series. For $r = 0$: $(xD)$ multiplies the $k$-th term of $C_0$ by $mk$, and
\[
\frac{mk}{m^{mk}(k!)^m}x^{mk} = \frac{x \cdot x^{mk-1}}{m^{mk-1}(k-1)!\,(k!)^{m-1}},
\]
which after the shift $k \mapsto k+1$ is $x$ times the $k$-th term of $C_1$. For $1 \le r \le m-1$: $(xD + r)$ multiplies the $k$-th term of $C_r$ by $(mk + m - r) + r = m(k+1)$, and
\[
\frac{m(k+1)}{m^{mk+m-r}\,(k!)^r\,((k+1)!)^{m-r}}\,x^{mk+m-r} = \frac{x\cdot x^{mk+m-(r+1)}}{m^{mk+m-(r+1)}\,(k!)^{r+1}\,((k+1)!)^{m-(r+1)}},
\]
which is $x$ times the $k$-th term of $C_{r+1}$ when $r + 1 \le m - 1$, and---since the case $r+1 = m$ reads $x\, x^{mk}/\big(m^{mk}(k!)^{m}\big)$---is $x$ times the $k$-th term of $C_0$ when $r = m-1$. The identification with ${}_0F_{m-1}$ follows by rewriting the factorials as Pochhammer symbols. The eigenvalue equation is \eqref{eq-carrier-ode}.
\end{proof}

In the notation of Delerue's hyper-Bessel functions $I^{(m-1)}_{\nu_1,\ldots,\nu_{m-1}}$ \cite{delerue1953, kiryakova1994}, we have $C^{(m)}_0 = I^{(m-1)}_{(0,\ldots,0)}$ and, for $1 \le r \le m-1$, $C^{(m)}_r = I^{(m-1)}_{(0,\ldots,0,1,\ldots,1)}$ with $r-1$ zeros and $m-r$ ones. The carrier expansion theorem now holds verbatim:

\begin{theorem}[General carrier expansion]\label{thm-general-carrier-expansion}
Let $(f_0, \ldots, f_{m-1})$ be any carrier system of modulus $m$ for $xD$. Then
\[
(xD)^n f_0(x) = \sum_{k=0}^n \mStirling{n}{k}_m\, x^k f_{k \bmod m}(x),
\qquad
f_0(xe^t) = \sum_{n=0}^\infty \frac{t^n}{n!}\sum_{k=0}^n \mStirling{n}{k}_m\, x^k f_{k \bmod m}(x),
\]
the second identity read formally or, for analytic $f_0$, as explained after Theorem \ref{thm-carrier-expansion}.
\end{theorem}
\begin{proof}
Identical to that of Theorem \ref{thm-carrier-expansion}: since $xD$ is a derivation, $(xD)(x^k f_r) = (k-r)x^kf_r + x^{k+1}f_{(r+1) \bmod m}$ for $r = k \bmod m$, and $k - r = m\floor{k/m}$; the coefficients therefore evolve by \eqref{eq-mstirling-recurrence}.
\end{proof}

For the hyper-Bessel system of Proposition \ref{prop-hyper-bessel} we retain the notation of Theorem \ref{thm-carrier-expansion} and write $V_n(x) = \sum_{k} \mStirling{n}{k}_m x^k\, C^{(m)}_{k \bmod m}(x)$. We also call
\[
\mathscr{B}^{(m)}_n(x) = \sum_{k=0}^{n} \mStirling{n}{k}_m\, x^k
\]
the \emph{$m$-Touchard polynomials}; by Corollary \ref{cor-mbell-rowsums}, $\mathscr{B}^{(m)}_n(1) = B^{(m)}_n$.

\subsection{A general Dobi\'nski formula}

Composing the primary hyper-Bessel carrier with the exponential produces a distinguished solution of the $m$-Bell ODE, whose Maclaurin coefficients generalize the Dobi\'nski series---and the coefficients of the carriers in those coefficients are precisely the primitive $m$-Bell numbers.

\begin{theorem}[Generalized Dobi\'nski formula]\label{thm-general-dobinski}
Fix $m \geq 1$ and write $C_r = C^{(m)}_r$. The function $\mathcal{A}(x) = C_0(e^x)$ satisfies $\mathcal{A}^{(m)} = e^{mx}\mathcal{A}$, and its Maclaurin coefficients
\[
S_m(n) := \sum_{k=0}^{\infty} \frac{(mk)^n}{\left(m^{k}\,k!\right)^{m}}
\]
satisfy
\[
S_m(n) = \sum_{r=0}^{m-1} B^{(m,r)}_n\, C_r(1),
\]
where $B^{(m,r)}$ are the primitive $m$-Bell sequences. For $m = 1$ this is Dobi\'nski's formula $\sum_k k^n/k! = e\,B_n$; for $m=2$ it is Theorem \ref{thm-bessel-dobiski}.
\end{theorem}
\begin{proof}
Expanding the series of $C_0$ at $e^x$ and swapping summation order (absolute convergence),
\[
C_0(e^x) = \sum_{k=0}^\infty \frac{e^{mkx}}{m^{mk}(k!)^m} = \sum_{n=0}^\infty \frac{x^n}{n!} \sum_{k=0}^\infty \frac{(mk)^n}{m^{mk}(k!)^m} = \sum_{n=0}^\infty S_m(n)\frac{x^n}{n!}.
\]
On the other hand, setting $x = 1$ in Theorem \ref{thm-general-carrier-expansion} and using $e^{txD}f(x)\big|_{x=1} = f(e^t)$:
\[
C_0(e^t) = \sum_{n=0}^\infty \frac{t^n}{n!}\sum_{k=0}^n \mStirling{n}{k}_m C_{k \bmod m}(1) = \sum_{n=0}^\infty \frac{t^n}{n!}\sum_{r=0}^{m-1}\rho_r(n)\,C_r(1),
\]
where $\rho_r(n)$ are the residue-class row sums. Comparing coefficients and applying Corollary \ref{cor-mbell-rowsums}(i) gives the result. That $\mathcal{A} = C_0(e^x)$ solves the ODE follows from Propositions \ref{prop-general-ode} and \ref{prop-hyper-bessel}.
\end{proof}

\begin{remark}\label{rem-solution-space}
The solution space of $\mathcal{A}^{(m)} = e^{mx}\mathcal{A}$ is $m$-dimensional and is spanned by the e.g.f.s of the $m$ primitive sequences. For $m = 2$ we exhibited this basis explicitly through $I_0$ and $K_0$ (Theorem \ref{thm-m2-egf}), and the Wronskian identity let us cancel the carriers and certify integer coefficients (Theorem \ref{thm-bb-polynomials}). For $m \geq 3$, the point $t=0$ of \eqref{eq-general-ode-t} is a regular singular point whose indicial polynomial has the single root $0$ with multiplicity $m$; hence $C_0(t)$ is the unique solution analytic at the origin (up to scaling), and the remaining $m-1$ Frobenius solutions involve powers of $\log t$. By analogy with the Bessel case $m = 2$, where the corresponding logarithmic solution is $K_0$, we call these Frobenius solutions \emph{hyper-Bessel functions of the second kind}; the general theory of hyper-Bessel operators and functions is developed in \cite{kiryakova1994}. Deriving explicit weights, analogous to those of Theorem \ref{thm-m2-egf}, that combine these singular solutions into the individual primitive e.g.f.s is left as an open problem; note that by Corollary \ref{cor-mbell-rowsums} the integrality of all coefficients is already established without this analytic machinery.
\end{remark}

\begin{remark}
The primitive/composite decomposition has a classical constant-coefficient shadow: replacing the operator equation $\mathrm{BINOM}^m \circ\, a = \mathrm{L}^m \circ\, a$ by $a = \mathrm{L}^m \circ\, a$ (i.e.\ dropping the transform) yields the ODE $\mathcal{A}^{(m)} = \mathcal{A}$, whose primitive solutions are the $m$-sections of the exponential function studied by Tauber \cite{tauberGeneralizationsExponentialFunction1960}. The $m$-Bell theory developed here can thus be viewed as the ``binomial-transform deformation'' of the theory of exponential sections.
\end{remark}

\subsection{Duals, factorials, and combinatorics for general $m$}

All the structures of Sections \ref{sec-combinatorial} and \ref{sec-first-kind} generalize from $m = 2$ to arbitrary $m$, with the same proofs. We record the statements.

\begin{definition}\label{def-general-duals}
For $m \geq 1$ define:
\begin{itemize}
    \item[(i)] the unsigned \emph{$m$-Stirling numbers of the first kind} by $\mStirlingone{0}{k}_m = \delta_{0,k}$ and
    \[
    \mStirlingone{n+1}{k}_m = m\floor{n/m}\,\mStirlingone{n}{k}_m + \mStirlingone{n}{k-1}_m;
    \]
    \item[(ii)] the \emph{$m$-falling and $m$-rising factorials}
    \[
    x^{\uuline{n}_m} = \prod_{k=0}^{n-1}\left(x - m\floor{k/m}\right), \qquad
    x^{\ooline{n}_m} = \prod_{k=0}^{n-1}\left(x + m\floor{k/m}\right),
    \]
    in which each factor repeats $m$ times (except possibly the last); when $m$ is clear from context we drop the subscript, as we did for $m=2$;
    \item[(iii)] the \emph{$m$-Lah numbers} $\mLah{n}{k}_m$ by $x^{\ooline{n}_m} = \sum_k \mLah{n}{k}_m\, x^{\uuline{k}_m}$.
\end{itemize}
\end{definition}

\begin{theorem}\label{thm-general-duals}
For every $m \geq 1$ and all $0 \le k \le n$, with weights $w_j = m\floor{j/m}$:
\begin{itemize}
\item[(i)] $x^{\ooline{n}_m} = \sum_k \mStirlingone{n}{k}_m x^k$, \quad $x^{\uuline{n}_m} = \sum_k (-1)^{n-k}\mStirlingone{n}{k}_m x^k$, \quad $x^n = \sum_k \mStirling{n}{k}_m x^{\uuline{k}_m}$;
\item[(ii)] the signed first-kind and second-kind triangles are inverse matrices;
\item[(iii)] $\mStirlingone{n}{k}_m = e_{n-k}(w_0, \ldots, w_{n-1})$ and $\mStirling{n}{k}_m = h_{n-k}(w_1, \ldots, w_k)$, and the column generating function of the second kind is $\sum_n \mStirling{n}{k}_m x^n = x^k\big/\prod_{j=1}^{k}(1 - w_j x)$;
\item[(iv)] $\mLah{n}{k}_m = \sum_j \mStirlingone{n}{j}_m \mStirling{j}{k}_m$, with recurrence
\[
\mLah{n+1}{k}_m = \left(m\floor{n/m} + m\floor{k/m}\right)\mLah{n}{k}_m + \mLah{n}{k-1}_m;
\]
\item[(v)] the row sums of the first kind equal $\prod_{j=0}^{n-1}\left(m\floor{j/m} + 1\right)$;
\item[(vi)] $\mStirlingone{n}{k}_m$ counts the permutations of $[n]$ with $k$ cycles built by the restricted insertion process in which element $i+1$ either opens a new cycle or is inserted after one of the $m\floor{i/m}$ smallest existing elements.
\end{itemize}
\end{theorem}
\begin{proof}
Identical, word for word, to the proofs of Theorems \ref{thm-bessel-factorial-expansion}, \ref{thm-inverse-expansion} and \ref{thm-bessel-lah}, Propositions \ref{prop-orthogonality}, \ref{prop-symmetric-functions} and \ref{prop-a000246} (first equality), and Theorem \ref{thm-insertion-model}, with $2\floor{\cdot/2}$ replaced by $m\floor{\cdot/m}$ throughout.
\end{proof}

\begin{theorem}[The $m$-compartment urn model]\label{thm-general-urn}
$\mStirling{n}{k}_m$ counts the arrangements of $n$ labeled marbles, placed one at a time, into $k$ compartments distributed across urns with $m$ unlabeled compartments each, under the constraints and indistinguishability conventions of Theorem \ref{thm-urn-model} (mutatis mutandis): within an urn, a compartment may receive a second marble only after every compartment of that urn is non-empty, and a new urn may be opened only when all previously used urns are completely non-empty. Consequently the composite $m$-Bell number $B^{(m)}_n$ counts all such arrangements of $n$ marbles, and the primitive $B^{(m,r)}_n$ counts those in which the number of used compartments is $\equiv r \pmod m$.
\end{theorem}
\begin{proof}
As in Theorem \ref{thm-urn-model}: if $k = mj + r$ with $0 \le r < m$ compartments are in use, then $j$ urns are complete and one urn has $r$ used compartments; the marbles in the incomplete urn's used compartments are frozen, so exactly $mj = m\floor{k/m}$ old compartments can receive the next marble, while opening the next compartment (in the incomplete urn if $r > 0$, or in a fresh urn if $r = 0$) is the unique alternative. The final two claims follow from Corollary \ref{cor-mbell-rowsums}.
\end{proof}

The succession rule of Appendix C also generalizes directly, encoding the recurrence as a generating tree:
\begin{definition}[Succession rule for $m$-Stirling numbers of the second kind]\label{def-succession-general}
\begin{align*}
\Omega_m: \begin{cases}
(0) \\
(mj+r) \to (mj+r)^{mj}\,,\ (mj+r+1) \qquad \text{for } j \geq 0,\ 0 \leq r < m,
\end{cases}
\end{align*}
i.e.\ a node labeled $(k)$ produces $m\floor{k/m}$ children with its own label and one child labeled $(k+1)$. The number of nodes labeled $(k)$ at level $n$ equals $\mStirling{n}{k}_m$, and the total number of nodes at level $n$ equals $B^{(m)}_n$.
\end{definition}
This representation further reinforces the structural unity of the $m$-Stirling family and provides an efficient way to generate and analyze these numbers through tree-based methods.

\section{Moments of the Conway--Maxwell--Poisson distribution}\label{sec-cmp}

\noindent Perhaps the most intriguing consequence of the present theory is a probabilistic one. The classical Bell--Stirling framework is inseparable from the Poisson distribution: if $X \sim \mathrm{Poisson}(y)$, then the factorial moments are $\mathbb{E}\left[X^{\underline{n}}\right] = y^n$, the raw moments are the Touchard polynomials $\mathbb{E}[X^n] = T_n(y)$, and at $y = 1$ Dobi\'nski's formula \eqref{eq-bell-dobinski} states precisely that $\mathbb{E}[X^n] = B_n$. In this section we show how the package generalizes to the \emph{Conway--Maxwell--Poisson} (CMP) distribution with integer dispersion parameter $\nu = m$: the $m$-falling moments are powers of the rate times a fixed carrier ratio, the raw moments are carrier-weighted analogues of the Touchard values, and at $y = 1$ the scaled moments are exactly the carrier-weighted combinations of the \emph{primitive} $m$-Bell sequences furnished by the generalized Dobi\'nski formula. For $m = 1$ every carrier ratio equals $1$ and the relation collapses to the classical one; for $m \geq 2$ the carriers weight the residue classes unequally, so it is the primitive package, rather than the single composite number $B^{(m)}_n$, that the moments encode.

The CMP distribution \cite{conwayMaxwell1962, shmueliUsefulDistributionFitting2005} is a two-parameter generalization of the Poisson distribution with probability mass function
\[
\mathbb{P}(X = k) = \frac{\lambda^k}{(k!)^{\nu}\, Z(\lambda, \nu)}, \qquad k = 0, 1, 2, \ldots, \qquad Z(\lambda,\nu) = \sum_{j=0}^{\infty}\frac{\lambda^j}{(j!)^{\nu}},
\]
where either $\nu > 0$ and $\lambda > 0$, or $\nu = 0$ and $0 < \lambda < 1$ (for $\nu = 0$ the normalizing series is geometric and diverges when $\lambda \geq 1$) \cite{dalyGaunt2016}. The case $\nu=1$ is the Poisson distribution; $\nu > 1$ models under-dispersion and $\nu < 1$ over-dispersion, which has made the distribution a popular tool in applied count modeling \cite{shmueliUsefulDistributionFitting2005}. Our first observation is that at integer $\nu = m$ and with the parametrization $\lambda = (y/m)^m$, the normalizing constant is exactly the primary hyper-Bessel carrier of Proposition \ref{prop-hyper-bessel}:
\begin{equation}\label{eq-cmp-normalizer}
Z\!\left(\left(\tfrac{y}{m}\right)^{m},\, m\right) = \sum_{j=0}^\infty \frac{(y/m)^{mj}}{(j!)^m} = C^{(m)}_0(y).
\end{equation}
For $m = 2$ this is the familiar fact that $Z(\lambda, 2) = I_0(2\sqrt{\lambda})$ \cite{shmueliUsefulDistributionFitting2005, dalyGaunt2016}.

It is well known (and immediate) that the CMP distribution has clean \emph{powered} factorial moments:
\begin{lemma}\label{lem-cmp-factorial}
If $X \sim \mathrm{CMP}(\lambda, \nu)$ with $\nu > 0$, then $\mathbb{E}\left[\left(X^{\underline{j}}\right)^{\nu}\right] = \lambda^j$ for all $j \geq 0$.
\end{lemma}
\begin{proof}
$\displaystyle \sum_{k\ge j}\left(\frac{k!}{(k-j)!}\right)^{\!\nu}\frac{\lambda^k}{(k!)^\nu Z} = \frac{\lambda^j}{Z}\sum_{i \ge 0}\frac{\lambda^{i}}{(i!)^\nu} = \lambda^j. $
\end{proof}

For $\nu = 1$, Lemma \ref{lem-cmp-factorial} is the Poisson factorial-moment identity $\mathbb{E}[X^{\underline{j}}] = \lambda^j$. The key discovery is that for $\nu = m$, the correct generalization of the falling factorial \emph{of the variable itself} (rather than a power of it) is precisely the $m$-falling factorial of Definition \ref{def-general-duals}, evaluated at $mX$---and its expectation picks up exactly one hyper-Bessel carrier ratio.

\begin{theorem}[$m$-falling moments of the CMP distribution]\label{thm-cmp-falling}
Let $m \geq 1$, $y > 0$, and $X \sim \mathrm{CMP}\!\left((y/m)^m, m\right)$. Then for all $n \geq 0$, with $r = n \bmod m$,
\[
\mathbb{E}\left[(mX)^{\uuline{n}_m}\right] = y^{n}\;\frac{C^{(m)}_{r}(y)}{C^{(m)}_0(y)}.
\]
In particular, for $m = 1$: $\mathbb{E}[X^{\underline{n}}] = y^n$ (Poisson); and for $m=2$:
\[
\mathbb{E}\left[(2X)^{\uuline{n}}\right] = \begin{cases} y^n & n \text{ even}, \\[2pt] y^n\, \dfrac{I_1(y)}{I_0(y)} & n \text{ odd}. \end{cases}
\]
\end{theorem}
\begin{proof}
Write $n = qm + r$ with $0 \le r < m$. In the product $(mX)^{\uuline{n}_m} = \prod_{k=0}^{n-1}(mX - m\floor{k/m})$, the factor $(mX - mi)$ appears $m$ times for each $0 \le i \le q - 1$ and the factor $(mX - mq)$ appears $r$ times, so
\[
(mX)^{\uuline{n}_m} = m^{n}\left(X^{\underline{q}}\right)^{m}(X - q)^{r}.
\]
Taking expectations with $\lambda = (y/m)^m$ and $Z = C_0(y)$ (by \eqref{eq-cmp-normalizer}):
\[
\mathbb{E}\left[(mX)^{\uuline{n}_m}\right] = \frac{m^n}{C_0(y)}\sum_{k \geq q} \left(\frac{k!}{(k-q)!}\right)^{m} (k-q)^r\, \frac{\lambda^k}{(k!)^m}
= \frac{m^n \lambda^q}{C_0(y)}\sum_{i \geq 0} \frac{i^r\,\lambda^{i}}{(i!)^m}.
\]
For $r = 0$ the remaining sum is $Z = C_0(y)$ and the result is $m^n\lambda^q = y^n$. For $1 \le r \le m-1$, use $i^r/(i!)^m = 1/\big(((i-1)!)^r\,(i!)^{m-r}\big)$ for $i \geq 1$ and shift $i = k+1$:
\[
\sum_{i \geq 1} \frac{i^r \lambda^i}{(i!)^m} = \sum_{k \geq 0}\frac{(y/m)^{m(k+1)}}{(k!)^r\,((k+1)!)^{m-r}} = \left(\frac{y}{m}\right)^{r} C_r(y),
\]
by the series in Proposition \ref{prop-hyper-bessel}. Hence the expectation equals $m^n \lambda^q (y/m)^r C_r(y)/C_0(y) = y^{mq+r}\,C_r(y)/C_0(y) = y^n\, C_r(y)/C_0(y)$.
\end{proof}

\begin{theorem}[Raw moments of the CMP distribution]\label{thm-cmp-moments}
With the assumptions of Theorem \ref{thm-cmp-falling},
\[
m^{n}\, \mathbb{E}\left[X^{n}\right] = \frac{1}{C^{(m)}_0(y)} \sum_{k=0}^{n} \mStirling{n}{k}_m\; y^{k}\; C^{(m)}_{k \bmod m}(y) = \frac{V_n(y)}{C^{(m)}_0(y)},
\]
where $V_n$ is the carrier-Touchard polynomial of Theorem \ref{thm-general-carrier-expansion} for the hyper-Bessel system. For $m = 1$ this is the classical $\mathbb{E}[X^n] = T_n(y)$, the Touchard polynomial evaluated at the rate.
\end{theorem}
\begin{proof}
Apply the basis conversion of Theorem \ref{thm-general-duals}(i) to the polynomial identity $(mX)^n = \sum_k \mStirling{n}{k}_m (mX)^{\uuline{k}_m}$, take expectations, and insert Theorem \ref{thm-cmp-falling}.
Alternatively and equivalently, $\mathbb{E}[(mX)^n] = \frac{1}{C_0(y)}\sum_k (mk)^n \lambda^k/(k!)^m = \frac{1}{C_0(y)}\left[(xD)^n C_0\right](y)$, and Theorem \ref{thm-general-carrier-expansion} expands the right-hand side.
\end{proof}

As a check, the case $n = 1$, $m = 2$ gives the known mean of the CMP distribution with $\nu = 2$ \cite{dalyGaunt2016}: $\mathbb{E}[X] = \tfrac{y}{2}\,I_1(y)/I_0(y) = \sqrt{\lambda}\; I_1(2\sqrt{\lambda})/I_0(2\sqrt{\lambda})$.

\begin{corollary}[Primitive $m$-Bell numbers as CMP moment coefficients]\label{cor-cmp-mbell}
Let $X \sim \mathrm{CMP}(m^{-m}, m)$, i.e.\ $y = 1$. Then
\[
m^{n}\,\mathbb{E}\left[X^{n}\right] = \sum_{r=0}^{m-1} B^{(m,r)}_n\; \frac{C^{(m)}_r(1)}{C^{(m)}_0(1)} = \frac{S_m(n)}{C^{(m)}_0(1)}.
\]
For $m=1$ ($X \sim \mathrm{Poisson}(1)$) this is $\mathbb{E}[X^n] = B_n$; for $m = 2$ ($X \sim \mathrm{CMP}(1/4, 2)$),
\[
2^{n}\, \mathbb{E}\left[X^{n}\right] = B^{(2,0)}_n + B^{(2,1)}_n\,\frac{I_1(1)}{I_0(1)} = \textup{A351143}(n) + \textup{A351028}(n)\cdot\frac{I_1(1)}{I_0(1)}.
\]
\end{corollary}
\begin{proof}
Set $y = 1$ in Theorem \ref{thm-cmp-moments} and group the terms by $k \bmod m$, applying Corollary \ref{cor-mbell-rowsums}; the second expression restates Theorem \ref{thm-general-dobinski}.
\end{proof}

Thus the scaled moments of the $\mathrm{CMP}(m^{-m}, m)$ distribution are $\mathbb{Z}$-linear combinations of the fixed ``Bessel quotients'' $C_r(1)/C_0(1)$, and the integer coefficients are exactly the primitive $m$-Bell numbers. (For $m \geq 2$ the composite number $B^{(m)}_n = \sum_r B^{(m,r)}_n$ is \emph{not} itself a CMP moment---already $2\,\mathbb{E}[X] = I_1(1)/I_0(1) \approx 0.4464$ while $B^{(2)}_1 = 1$---it is the coefficient package that the moments carry.) In this precise sense, the Conway--Maxwell--Poisson distribution is the natural probabilistic home of the $m$-Bell numbers, just as the Poisson distribution is for the Bell numbers. It seems fitting that a distribution introduced for queuing systems with state-dependent service rates \cite{conwayMaxwell1962}, and revived for fitting under- and over-dispersed count data \cite{shmueliUsefulDistributionFitting2005}, turns out to encode a natural generalization of the partition-counting apparatus of enumerative combinatorics.

\section{Final remarks}\label{sec-final-remarks}

\noindent We conclude with remarks situating the $m$-Stirling framework within the literature, and a list of open problems. Appendix D presents a reference table cataloguing the parallels between the classical theory, the Bessel case $m = 2$, and general $m$.

\subsection{Relation to other generalizations of Stirling numbers}
Triangular arrays defined by two-term recurrences of the form $T(n+1,k) = f(n,k)T(n,k) + T(n,k-1)$ have been studied in great generality---as ``Galton arrays'' by Neuwirth \cite{neuwirth_recursively_2001}, through bivariate generating function PDEs by Barbero et al.\ \cite{barbero_g_bivariate_2014}, and via explicit-formula machinery by Spivey \cite{spiveySolutionsGeneralCombinatorial2011} and Lancaster, Mansour, Mulay and Shattuck \cite{mansour_general_2012}. The $m$-Stirling triangles fit these frameworks formally but sit outside their sharpest results: the coefficient $m\floor{k/m}$ is not of the affine form $\alpha + \beta n + \gamma k$ covered by the complete classification in \cite{barbero_g_bivariate_2014}, and the general explicit formula of \cite{mansour_general_2012} requires the column weights $b_k$ to be pairwise distinct, whereas ours occur with multiplicity $m$---this is exactly why the partial-fraction analysis in Appendix A involves higher-order poles and produces the harmonic-number corrections in Theorem \ref{thm-bs-formula}. The $m$-Stirling numbers are also distinct from Broder's $r$-Stirling numbers \cite{broder1984} (which restrict where the first $r$ elements sit) and from the Bessel numbers of Cheon et al.\ \cite{cheonGeneralizedBesselNumbers2013a} (arrays arising from the coefficients of Bessel polynomials); the appearance of Bessel functions in our $m = 2$ case is structural (the e.g.f.\ ODE), not definitional.

\subsection{The sequence A000246 and a missing bijection}
Proposition \ref{prop-a000246} shows that the row sums of the Bessel--Stirling triangle of the first kind count the permutations of odd order, while Theorem \ref{thm-insertion-model} exhibits the triangle itself as counting parity-admissible permutations by cycles. Since the two $9$-element sets for $n = 4$ already differ, the equinumerosity is genuinely non-trivial at the level of the statistic. An \emph{explicit} bijection between odd-order permutations and parity-admissible permutations---one translating a natural statistic on odd-order permutations into the cycle count of the admissible ones---would give a fully combinatorial account of Table \ref{tab-bessel-stirling1}.

\subsection{Open problems}
\begin{problem}
Find a direct bijective proof of the shift identity \eqref{eq-shift-identity} in the urn model of Theorem \ref{thm-general-urn}: the right-hand side suggests distinguishing, in a placement of $n+m$ marbles, the $j$ marbles that avoid the ``last'' urn, with the factor $m^{n-j}$ recording an $m$-fold choice for each of the remaining marbles. Such a proof would make the ``pairing effect'' heuristic of Section \ref{sec-combinatorial} precise.
\end{problem}
\begin{problem}
Determine explicit closed forms, analogous to Theorem \ref{thm-m2-egf}, for the e.g.f.s of the individual primitive $m$-Bell sequences when $m \geq 3$, in terms of hyper-Bessel functions of the second kind (Remark \ref{rem-solution-space}).
\end{problem}
\begin{problem}
The Bell numbers satisfy the Touchard congruence $B_{p+n} \equiv B_n + B_{n+1} \pmod p$ for every prime $p$. The $m$-Stirling triangles show visible $m$-adic structure (for instance $\mStirling{n}{m}_m = m^{\,n-m}$ for $n \geq m$, by the column generating function); a systematic study of congruence properties of the $m$-Bell sequences is open.
\end{problem}
\begin{problem}
The Hankel determinants of the Bell numbers are the superfactorials $\prod_{i=0}^{n} i!$. The Hankel determinants of A007472 begin $1, 0, -4, -48, 0, 442368, \ldots$, with zeros and negative entries, so A007472 is not the moment sequence of a positive measure---yet Corollary \ref{cor-cmp-mbell} exhibits it as the integer ``shadow'' of one. Explaining this Hankel structure (e.g.\ via continued fractions or orthogonal polynomials for the CMP weight) is open.
\end{problem}
\begin{problem}
Develop $q$-analogues: both the urn model (Theorem \ref{thm-general-urn}) and the file-placement interpretation (Proposition \ref{prop-symmetric-functions}) suggest natural inversion statistics, and the Garsia--Remmel $q$-rook framework \cite{garsiaRemmel1986} applies to the doubled staircase boards.
\end{problem}
\begin{problem}
Determine the asymptotics of $B^{(m)}_n$ for fixed $m \geq 2$, generalizing the Moser--Wyman asymptotics of the Bell numbers \cite{moserWyman1955}; the Dobi\'nski-type series of Theorem \ref{thm-general-dobinski} is the natural starting point for a saddle-point analysis, though it delivers the carrier-weighted combination $S_m(n) = \sum_r B^{(m,r)}_n C_r(1)$, from which the asymptotics of the individual primitives (and hence of $B^{(m)}_n$) must still be extracted.
\end{problem}

\printbibliography

\pagebreak

\section*{Appendix A: Explicit formula for the Bessel--Stirling numbers (proof of Theorem \ref{thm-bs-formula})}
Let $A_k$ be the ordinary generating function for the $k$-th column over $x^n$:
\[
A_k(x) = \sum_{n\ge0} \mStirling{n}{k}_2 x^n.
\]
We first multiply each side of the recurrence \eqref{eq-bs-recurrence} by $x^n$ and sum over $n$:
\[
\begin{aligned}
\sum_{n \geq 1} \mStirling{n}{k}_2x^n &= 2\floor{k/2}\sum_{n \geq 1}\mStirling{n-1}{k}_2x^n + \sum_{n \geq 1} \mStirling{n-1}{k-1}_2 x^n \\
A_k -\mStirling{0}{k}_2&=2\floor{k/2}x A_k+x A_{k-1}, \quad k \geq 1.
\end{aligned}
\]
Since $\mStirling{0}{k}_2 = 0$ when $k \neq 0$, this simplifies to the following recurrence for the column generating function:
\[
A_k = \frac{x}{(1-2\floor{k/2}x)} A_{k-1}.
\]
We know from the initial conditions that $A_0(x) = 1$, hence:
\[
\begin{aligned}
    A_0 &= 1 \\
    A_1 &= x \\
    A_2 &= \frac{x^2}{1-2x} \\
    A_3 &= \frac{x^3}{(1-2x)^2} \\
    A_4 &= \frac{x^4}{(1-2x)^2 (1-4x)} \\
    A_5 &= \frac{x^5}{(1-2x)^2 (1-4x)^2}, 
\end{aligned}
\]
and so on. More generally, the ordinary generating function for the $k$-th column is
\begin{equation}
A_k(x) = \frac{x^k}{\prod_{j=1}^k(1-2\floor{j/2}x)},
\end{equation}
which is Equation \eqref{eq-column-ogf} of the main text. In order to get an explicit formula for $\mStirling{n}{k}_2$, use the technique on p.~19 of \cite{wilfGeneratingfunctionologyThirdEdition2005}:
\[
\mStirling{n}{k}_2 = [x^n]A_k(x),
\]
where $[x^n]$ is the coefficient extraction operator. Start with a partial fraction decomposition of $A_k(x)$. First, denote
\[
A_k(x) = x^kP_k(x),
\]
where
\[
P_k(x) = \prod_{1 \leq j \leq k}\frac{1}{1-2\floor{j/2}x}.
\]
Hence
\begin{equation}\label{znk-coef}
\mStirling{n}{k}_2 = [x^n]x^kP_k(x) = [x^{n-k}]P_k(x).
\end{equation}
We can express $P_k(x)$ as a sum of partial fractions with unknown coefficients, writing $k = 2\mu$ or $k = 2\mu + 1$ according to parity (the letter $m$ is reserved for the paper's global modulus, equal to $2$ throughout this appendix):

\begin{equation}\label{pkx}
    \begin{aligned}
    P_{2\mu}(x) &= \sum_{1 \leq j \leq \mu} \frac{c_{2\mu,j}}{1-2jx} + \sum_{1 \leq j \leq \mu-1} \frac{d_{2\mu,j}}{(1-2jx)^2} \\
    P_{2\mu+1}(x) &= \sum_{1 \leq j \leq \mu} \frac{c_{2\mu+1,j}}{1-2jx} + \sum_{1 \leq j \leq \mu} \frac{d_{2\mu+1,j}}{(1-2jx)^2}
    \end{aligned}
\end{equation}

It is easy to calculate the $d$ coefficients with the standard Heaviside cover-up method; note that the factor $(1-2rx)^2$ cancels the two factors of $P_k$ with $j \in \{2r, 2r+1\}$, both of which have $\floor{j/2} = r$:
\[
\begin{aligned}
    d_{k,r} &= (1-2rx)^2P_k(x)\big|_{x=\frac{1}{2r}} \\
    &= (1-2rx)^2 \prod_{1 \leq j \leq k} \frac{1}{1-2 \floor{j/2}x} \bigg|_{x=\frac{1}{2r}} \\
    &= \prod_{1 \leq j < 2r} \frac{1}{1-\floor{j/2}/r} \prod_{2r+1 < j \leq k} \frac{1}{1- \floor{j/2}/r} \\
    &= \prod_{1 \leq j < 2r} \frac{r}{r-\floor{j/2}} \prod_{2r+1 < j \leq k} \frac{r}{r- \floor{j/2}},
\end{aligned}
\]
which simplifies to
\begin{equation}\label{d-k-r}
        d_{k,r}=\frac{(-r)^{k-1}}{r!\,r!\,(\floor{\frac{k}{2}}-r)!\,(\floor{\frac{k-1}{2}}-r)!}.
\end{equation}
$d_{k,r}$ varies only slightly for even and odd $k$:
\[
\begin{aligned}
    d_{2\mu,r} &= \frac{-r^{2\mu-1}}{r!\,r!\,(\mu-r)!\,(\mu-1-r)!} \\
    d_{2\mu+1,r} &= \frac{r^{2\mu}}{r!\,r!\,(\mu-r)!\,(\mu-r)!}
\end{aligned}
\]
For the $c_{k,j}$ coefficients, we have two main cases to consider. When $k=2\mu$ is even and $j=\mu$, $P_k(x)$ has a simple pole at $x = \frac{1}{2\mu}$, so we get $c_{2\mu,\mu}$ easily by the same method:
\begin{equation}\label{c-2m-m}
    c_{2\mu,\mu} = (1-2\mu x)P_{2\mu}(x)\big|_{x=\frac{1}{2\mu}} = \frac{\mu^{2\mu}}{\mu!\,\mu!}.
\end{equation}
For the remaining $c$ coefficients we have a pole of $P_k(x)$ of multiplicity 2 at $x = \frac{1}{2j}$, so we need to use the derivative method. Specifically:
\[
-2r\,c_{k,r} = \frac{\mathrm{d}}{\mathrm{d}x} (1-2rx)^2P_k(x)\bigg|_{x=\frac{1}{2r}}
\]
Let
\[
\begin{aligned}
Q_r &= (1-2rx)^2 P_k(x) = \prod_{j} \frac{1}{(1-2jx)^{e_j}},\\
e_j &=
    \begin{cases}
    2,& \mathrm{if}\quad 1 \leq 2j < k \land  j \neq r \\
    1, & \mathrm{if} \quad 2j = k \land j \neq r \\
    0, & \mathrm{otherwise}
    \end{cases}
\end{aligned}
\]
Then $c_{k,r} = -\frac{1}{2r}Q'_r\left (\frac{1}{2r} \right)$, and we can use the logarithmic derivative to simplify the calculation:

\[
\begin{aligned}
    \frac{Q'_{r}}{Q_r} &= \frac{\mathrm{d}}{\mathrm{d}x}\log(Q_r) = \frac{\mathrm{d}}{\mathrm{d}x}\left[ -\sum_{j} e_j\log(1-2jx) \right]=  \sum_{j} 2e_j\frac{j}{1-2jx}  \\
    c_{k,r} &= - \frac{1}{2r}Q'_r\left (\frac{1}{2r} \right) = -\frac{1}{2r} Q_r\left(\frac{1}{2r}\right) 2\sum_{j} e_j\frac{j}{1-j/r} \\
    &= -\frac{1}{2r} 2 d_{k,r}\sum_{j}e_j\frac{rj }{r-j} = - d_{k,r}\sum_{j}e_j \frac{j}{r-j} 
\end{aligned}
\]
Let $s_{k,r}=\sum_{j}e_j\frac{j}{r-j}$ and simplify. First, consider the case $k=2\mu+1$:
\[
\begin{aligned}
    s_{2\mu+1,r} &=  2\sum_{\substack{1 \leq j \leq \mu \\ j \neq r}}\frac{j}{r-j} \\
    &= 2\left[\frac{1}{r-1}+\frac{2}{r-2} + \cdots  + \frac{r-1}{1}-\frac{r+1}{1}-\frac{r+2}{2}-\cdots-\frac{\mu}{\mu-r}\right]\\
    &= 2\sum_{1 \leq j \leq r-1}\frac{r-j}{j} -2\sum_{1 \leq j \leq \mu-r}\frac{j+r}{j} \\
    &= 2r\left(\sum_{1 \leq j \leq r-1}\frac{1}{j}\right) - 2(r-1) - 2r\left(\sum_{1 \leq j \leq \mu-r}\frac{1}{j}\right)-2(\mu-r)\\
    &= 2r(H_{r-1}-H_{\mu-r})-2\mu+2,
\end{aligned}
\]
where $H_n$ is the $n$-th harmonic number. We proceed similarly for $k=2\mu$, with some care about $e_{\mu}=1$:
\[
\begin{aligned}
    s_{2\mu,r} &= \sum_{j}e_j \frac{j}{r-j} = \frac{\mu}{r-\mu} + 2\sum_{\substack{1 \leq j \leq \mu-1 \\ j \neq r}}\frac{j}{r-j}\\
    &= \frac{\mu}{r-\mu} + 2 \sum_{1 \leq j \leq r-1} \frac{j}{r-j} - 2\sum_{r+1 \leq j \leq \mu-1}\frac{j}{j-r} \\
    &= \frac{\mu}{r-\mu} + 2\sum_{1 \leq j \leq r-1}\frac{r-j}{j} -2\sum_{1 \leq j \leq \mu-r-1}\frac{j+r}{j} \\
    &= \frac{\mu}{r-\mu} + 2rH_{r-1}-2(r-1)-2rH_{\mu-r-1}-2(\mu-r-1) \\
    &=  2rH_{r-1} - 2rH_{\mu-r-1}-2(\mu-2)-\frac{\mu}{\mu-r}  \\
    &= 2rH_{r-1} - 2rH_{\mu-r-1} - 2(\mu-2) - \frac{r}{\mu-r} -1 \\
    &= r\left(2H_{r-1} - H_{\mu-r-1}-H_{\mu-r}\right) - 2\mu + 3.
\end{aligned}
\]
Generally, for $1 \leq r < k/2$ (this covers all the double poles; note that for odd $k = 2\mu+1$ it includes $r = \mu$):
\begin{equation}\label{s-k}
    s_{k,r} = r(2H_{r-1}-H_{\floor{k/2}-r}-H_{\floor{(k-1)/2}-r})-k+3.
\end{equation}
Combining everything we have:
\begin{equation}\label{d-s-c}
    \begin{aligned}
        d_{k,r} &= \frac{(-r)^{k-1}}{r!\,r!\,(\floor{k/2}-r)!\,(\floor{(k-1)/2}-r)!} \\
        s_{k,r} &= r(2H_{r-1}-H_{\floor{k/2}-r}-H_{\floor{(k-1)/2}-r})-k+3 \\
        c_{k,r} &=
        \begin{cases}
            - d_{k,r}\, s_{k,r} & \text{if}\, 1 \leq r < k/2\\
            \frac{r^{2r}}{r!\,r!} & \text{if}\, r = k/2\\
            0 & \text{otherwise}
        \end{cases}
    \end{aligned}
\end{equation}
Next, let us find an explicit formula for $\mStirling{n}{k}_2$; assume first that $n \geq k$. From \eqref{znk-coef} and \eqref{pkx}:
\[
\begin{aligned}
    \mStirling{n}{k}_2 &=[x^{n-k}]P_k(x) \\
    &= [x^{n-k}]\left[ \sum_{1 \leq j \leq \floor{k/2}} \frac{c_{k,j}}{1-2jx} + \sum_{1 \leq j \leq \floor{(k-1)/2}} \frac{d_{k,j}}{(1-2jx)^2} \right]\\
     &= \sum_{1 \leq j \leq \floor{k/2}}c_{k,j}[x^{n-k}]  \frac{1}{1-2jx} + \sum_{1 \leq j \leq \floor{(k-1)/2}} d_{k,j}[x^{n-k}]\frac{1}{(1-2jx)^2} 
\end{aligned}
\]
We have the standard geometric series expansions (the extraction of the power $x^{n-k}$ is legitimate because $n - k \geq 0$; the range $n < k$ is treated separately below):
\[
\begin{aligned}
     & [x^{n-k}]\frac{1}{1-2jx} = [x^{n-k}]\sum_{i\geq0}(2j)^ix^i = (2j)^{n-k}\\
     & [x^{n-k}]\frac{1}{(1-2jx)^2} = [x^{n-k}]\sum_{i\geq0}(i+1)(2j)^ix^i = (n-k+1)(2j)^{n-k}.
\end{aligned}
\]
This yields the finite formula of Theorem \ref{thm-bs-formula}: abbreviating
\begin{equation}\label{z-d-c}
        F_k(n) := \sum_{1 \leq j \leq \floor{\frac{k}{2}}}(2j)^{n-k}\, c_{k,j} \;+\; (n-k+1)\sum_{1 \leq j \leq \floor{\frac{k-1}{2}}}(2j)^{n-k}\, d_{k,j},
\end{equation}
with $d_{k,r}$ and $c_{k,r}$ as in \eqref{d-s-c}, we have shown that $\mStirling{n}{k}_2 = F_k(n)$ for $n \geq k \geq 2$. The two sums have the same range when $k$ is odd, but for even $k$ the endpoint $j = k/2$ is a \emph{simple} pole of $P_k$: it contributes a $c$-term only, and the $d$-sum must stop at $\floor{(k-1)/2}$ (indeed \eqref{d-k-r} would involve $(-1)!$ at $r = k/2$).

\subsection*{The range $2 \leq n < k$}
The coefficient extraction above requires $n \geq k$, while Theorem \ref{thm-bs-formula} asserts the formula for all $n, k \geq 2$; we must therefore show that $F_k(n) = 0$ for $2 \leq n < k$, matching the triangular zeros of the array. Expand $P_k$ as a Laurent series at infinity: for $|x| > \tfrac{1}{2}$,
\[
\frac{1}{1-2jx} = -\sum_{s \geq 1}(2j)^{-s}x^{-s}, \qquad \frac{1}{(1-2jx)^2} = \sum_{s \geq 2}(s-1)(2j)^{-s}x^{-s},
\]
so the partial-fraction decomposition \eqref{pkx} gives
\[
\begin{aligned}
P_k(x) &= \sum_{s \geq 1}\Bigg[-\sum_{1 \leq j \leq \floor{\frac{k}{2}}}c_{k,j}(2j)^{-s} + (s-1)\sum_{1 \leq j \leq \floor{\frac{k-1}{2}}}d_{k,j}(2j)^{-s}\Bigg]x^{-s} \\
&= -\sum_{s \geq 1} F_k(k-s)\, x^{-s},
\end{aligned}
\]
where the second equality substitutes $n = k - s$ in \eqref{z-d-c} and uses $s - 1 = -(n-k+1)$. On the other hand, $P_k$ is the reciprocal of a polynomial of degree $k - 1$ (its $k-1$ nontrivial factors $1 - 2\floor{j/2}x$, $2 \leq j \leq k$, each have degree one), so $P_k(x) = O(|x|^{-(k-1)})$ as $|x| \to \infty$, and the coefficients of $x^{-1}, \ldots, x^{-(k-2)}$ must all vanish. Since $s$ running through $1, \ldots, k-2$ corresponds to $n = k-s$ running through $k-1, \ldots, 2$, this proves
\[
F_k(n) = 0 \qquad \text{for } 2 \leq n \leq k-1,
\]
completing the proof of the finite formula of Theorem \ref{thm-bs-formula} for all $n, k \geq 2$.

Carrying the same expansion one order further produces two boundary constants needed below. Write
\[
A = \floor{k/2}, \qquad B = \floor{(k-1)/2}, \qquad\text{so that } A + B = k - 1 \ \ (k \geq 1).
\]
The multiset of weights $\{\floor{j/2} : 2 \leq j \leq k\}$ consists of each of $1, \ldots, B$ twice, together with $A$ once more when $k$ is even; hence $\prod_{j=2}^{k}\floor{j/2} = A!\,B!$ and $\sum_{j=2}^{k}\floor{j/2}^{-1} = H_A + H_B$. Writing each factor of $P_k$ as $\frac{1}{1-2wx} = \frac{-1}{2wx}\big(1 + \frac{1}{2wx} + O(x^{-2})\big)$ and multiplying over the $k-1$ weights,
\[
P_k(x) = \frac{(-1)^{k-1}}{2^{k-1}\,A!\,B!}\;x^{-(k-1)}\left(1 + \frac{H_A + H_B}{2}\,x^{-1} + O(x^{-2})\right),
\]
and comparing with $P_k(x) = -\sum_{s} F_k(k-s)\,x^{-s}$ at $s = k-1$ and $s = k$:
\begin{equation}\label{eq-laurent-boundary}
F_k(1) = \frac{(-1)^{k}}{2^{k-1}\,A!\,B!}, \qquad F_k(0) = \frac{(-1)^{k}\,(H_A + H_B)}{2^{k}\,A!\,B!} \qquad (k \geq 2).
\end{equation}

\subsection*{The unified limit representation}
The piecewise definition of $c_{k,r}$ in \eqref{d-s-c} is forced by the singularities of the harmonic numbers at negative integer indices, which occur when $r \geq k/2$ or $r \leq 0$. The singularities can be tamed by replacing harmonic numbers via the identity $H_n = \psi(n+1)+\gamma$, where $\gamma$ is the Euler--Mascheroni constant and $\psi$ the digamma function, and factorials by gamma functions: both $\psi$ and $\Gamma$ have simple poles at the non-positive integers, so appropriate ratios have finite limits there. The only expansions we shall need are the classical ones at the origin,
\[
\Gamma(z) = \frac{1}{z} - \gamma + O(z), \qquad \psi(z) = -\frac{1}{z} - \gamma + O(z) \qquad (z \to 0),
\]
together with the elementary values $\psi(j) = H_{j-1} - \gamma$ and $\Gamma(j+1) = j!$ at positive integers $j$.

With $A = \floor{k/2}$ and $B = \floor{(k-1)/2}$ as above, define for integers $n, k \geq 0$ and $0 \leq j \leq A$
\[
\begin{aligned}
T_{n,k}(j) &= \lim_{t \to j} \left[t^{n-1}\,
        \frac
            {n-2-t\,\Psi_k(t)}
            {\Gamma(t+1)^{2}\, \Gamma\left(A-t+1\right) \Gamma\left(B-t+1\right)}\right],
        \\
        \Psi_k(t) &= 2\psi(t)-\psi\left(A-t+1\right)-\psi\left(B-t+1\right).
\end{aligned}
\]
The limit representation of Theorem \ref{thm-bs-formula} asserts that
\begin{equation}\label{eq-limit-claim}
\mStirling{n}{k}_2 = (-1)^{k-1}\, 2^{n-k} \sum_{j=0}^{A} T_{n,k}(j) \qquad \text{for all integers } n, k \geq 0.
\end{equation}
We now evaluate every limit; all four cases are elementary.

\emph{(i) Interior points $1 \leq j < k/2$.} All gamma and digamma arguments are positive integers, so no limiting process is needed:
\[
T_{n,k}(j) = j^{n-1}\,\frac{n - 2 - j\big(2H_{j-1} - H_{A-j} - H_{B-j}\big)}{(j!)^2\,(A-j)!\,(B-j)!},
\]
the three Euler constants in $\Psi_k(j)$ cancelling. Since $d_{k,j} = (-1)^{k-1}j^{k-1}\big/\big((j!)^2\,(A-j)!\,(B-j)!\big)$ by \eqref{d-k-r}, and $n - 2 - j\,\Psi_k(j) = (n-k+1) - s_{k,j}$ by \eqref{s-k},
\[
(-1)^{k-1}2^{n-k}\,T_{n,k}(j) = (2j)^{n-k}\,d_{k,j}\,\big(n-k+1-s_{k,j}\big) = (2j)^{n-k}\big(c_{k,j} + d_{k,j}(n-k+1)\big):
\]
the interior limit terms reproduce exactly the corresponding terms of \eqref{z-d-c}---and they do so for \emph{every} integer $n \geq 0$, the computation being a polynomial identity in $n$.

\emph{(ii) The endpoint $j = A = k/2$ for even $k$.} Set $t = A - \epsilon$ with $\epsilon \to 0$. Since $B = A - 1$, we have $\Gamma(B-t+1) = \Gamma(\epsilon) = \epsilon^{-1}\big(1 + O(\epsilon)\big)$ and $t\,\psi(B-t+1) = t\,\psi(\epsilon) = -t/\epsilon + O(1)$, while every other factor stays regular. Hence
\[
T_{n,k}(A) = \lim_{\epsilon \to 0}\, A^{n-1}\,\frac{-A/\epsilon + O(1)}{(A!)^2\,\Gamma(1+\epsilon)\cdot\epsilon^{-1}\big(1+O(\epsilon)\big)} = -\frac{A^{n}}{(A!)^2}.
\]
Multiplying by the prefactor $(-1)^{k-1}2^{n-k} = -2^{n-k}$ gives $2^{n-k}A^{n}/(A!)^2 = (2A)^{n-k}\,c_{k,A}$ by \eqref{c-2m-m} (using $2A = k$): the endpoint limit reproduces the $c$-term of the simple pole and correctly contributes no $d$-term---again for every integer $n \geq 0$.

\emph{(iii) The point $j = 0$, for $k \geq 1$.} As $t \to 0$ we have $2t\,\psi(t) = -2 - 2\gamma t + O(t^2)$ and $\psi(A-t+1) + \psi(B-t+1) = H_A + H_B - 2\gamma + O(t)$, so the numerator satisfies
\[
n - 2 - t\,\Psi_k(t) = n + (H_A + H_B)\,t + O(t^2),
\]
while the denominator tends to $\Gamma(1)^2\,\Gamma(A+1)\,\Gamma(B+1) = A!\,B! \neq 0$. Therefore
\begin{equation}\label{eq-j0-limits}
T_{n,k}(0) = \lim_{t \to 0}\, t^{n-1}\,\frac{n + (H_A+H_B)\,t + O(t^2)}{A!\,B!\,\big(1+O(t)\big)} =
\begin{cases}
\dfrac{H_A + H_B}{A!\,B!} & n = 0,\\[8pt]
\dfrac{1}{A!\,B!} & n = 1,\\[6pt]
0 & n \geq 2.
\end{cases}
\end{equation}

\emph{(iv) The column $k = 0$.} Here $A = 0$ and $B = -1$, the sum in \eqref{eq-limit-claim} reduces to the single term $j = 0$, and case (iii) does not apply because $\Gamma(B-t+1) = \Gamma(-t)$ is itself singular. Expanding, $\Gamma(-t) = -t^{-1}\big(1 + \gamma t + O(t^2)\big)$ and $t\,\psi(-t) = 1 - \gamma t + O(t^2)$, so the numerator is $n - 2 - t\,\Psi_0(t) = (n - 2) + 2 + 1 + O(t^2) = n + 1 + O(t^2)$ while the denominator is $-t^{-1}\big(1 + O(t)\big)$. Hence $T_{n,0}(0) = \lim_{t\to0}\big[{-(n+1)}\,t^{n}\,\big(1+O(t)\big)\big] = -\delta_{n,0}$, and with the prefactor $(-1)^{0-1}\,2^{n-0} = -2^{n}$ the right-hand side of \eqref{eq-limit-claim} equals $\delta_{n,0} = \mStirling{n}{0}_2$, as required.

It remains to assemble the cases. For $k = 1$ (where $A = B = 0$ and only $j = 0$ occurs), \eqref{eq-j0-limits} and the prefactor $2^{n-1}$ give the values $0, 1, 0, 0, \ldots$ for $n = 0, 1, 2, \ldots$, i.e.\ $\delta_{n,1} = \mStirling{n}{1}_2$. For $k \geq 2$, cases (i) and (ii) show that
\[
(-1)^{k-1}\, 2^{n-k} \sum_{j=0}^{A} T_{n,k}(j) = (-1)^{k-1}\,2^{n-k}\,T_{n,k}(0) + F_k(n) \qquad (n \geq 0),
\]
and three regimes finish the proof:
\begin{itemize}
    \renewcommand\labelitemi{--}
    \item $n \geq 2$: the $j = 0$ term vanishes by \eqref{eq-j0-limits}, and $F_k(n) = \mStirling{n}{k}_2$---for $n \geq k$ by the coefficient extraction, and for $2 \leq n < k$ because both sides are $0$ by the Laurent computation above.
    \item $n = 1$: by \eqref{eq-j0-limits} and \eqref{eq-laurent-boundary} the two contributions are $(-1)^{k-1}2^{1-k}/(A!\,B!)$ and $(-1)^{k}2^{1-k}/(A!\,B!)$; they cancel, giving $0 = \mStirling{1}{k}_2$.
    \item $n = 0$: the two contributions are those of the previous case, each multiplied by $(H_A + H_B)/2$; they cancel again, giving $0 = \mStirling{0}{k}_2$.
\end{itemize}
This proves \eqref{eq-limit-claim} in full: the limit representation is valid for all integers $n, k \geq 0$, with the $j = 0$ term supplying precisely the corrections needed at $n \in \{0, 1\}$ and the Laurent identities \eqref{eq-laurent-boundary} enforcing the triangular zeros. \hfill $\square$

\newpage
\section*{Appendix B: Bessel--Stirling number identities}
\noindent In all cases below $n, k \geq 0$ and $j \geq 1$ are integers, $b = 2\floor{k/2}$ is used to simplify the notation (with the convention $b^0 = 1$ even when $b = 0$), and empty sums are $0$.
\begin{lemma}[Distant hockey-stick-like identity]\label{lemma-dist-hockey}
\[
    \mStirling{n+j}{k}_2 = b^j \mStirling{n}{k}_2 +\sum_{i=0}^{j-1}b^{j-1-i}\mStirling{n+i}{k-1}_2
\]
\end{lemma}
\begin{proof}
The base case for $j=1$ is the basic recurrence \eqref{eq-bs-recurrence}:
\[
    \mStirling{n+1}{k}_2 = b \,\mStirling{n}{k}_2 + \mStirling{n}{k-1}_2.
\]
Assume the identity holds for \textit{j} (induction hypothesis). Then from the basic recurrence we have
\[
    \mStirling{n+j+1}{k}_2 = b \, \mStirling{n+j}{k}_2 +\mStirling{n+j}{k-1}_2.
\]
From the induction hypothesis:
\[
    \begin{aligned}
        \mStirling{n+j+1}{k}_2 &= b \, \bigg(b^j \mStirling{n}{k}_2 +\sum_{i=0}^{j-1}b^{j-1-i}\mStirling{n+i}{k-1}_2\bigg) +\mStirling{n+j}{k-1}_2 \\
                   &= b^{j+1}\mStirling{n}{k}_2+\sum_{i=0}^jb^{j-i}\, \mStirling{n+i}{k-1}_2,
    \end{aligned}
\]
which completes the proof.
\end{proof}
\begin{proposition}[Hockey-stick-like identity]
For $k \geq 1$,
    \[
        \mStirling{n}{k}_2 = \sum_{j=k}^{n}b^{n-j} \mStirling{j-1}{k-1}_2.
    \]
\end{proposition}
\begin{proof}
Set $n=0$ in Lemma \ref{lemma-dist-hockey}:
\[
    \mStirling{j}{k}_2 = b^j \mStirling{0}{k}_2 +\sum_{i=0}^{j-1}b^{j-1-i}\mStirling{i}{k-1}_2.
\]
Due to the initial conditions $\mStirling{0}{0}_2=1$, $\mStirling{n}{0}_2 = \mStirling{0}{k}_2 =0$ ($n, k \geq 1$), the first term on the right-hand side vanishes for $k \geq 1$. Then we can simplify, rewriting the indices:
\[
\mStirling{n}{k}_2 = \sum_{j=0}^{n-1}b^{n-j-1}\mStirling{j}{k-1}_2,
\]
and since $\mStirling{n}{k}_2 =0$ when $k > n$, we can start $j$ from $k$ and shift it by 1 to match the proposition's form:
\[
        \mStirling{n}{k}_2 = \sum_{j=k}^{n}b^{n-j} \mStirling{j-1}{k-1}_2. \qedhere
\]
\end{proof}

Both identities, with $b = m\floor{k/m}$, hold verbatim for the general $m$-Stirling numbers of the second kind, by the same proofs.

\pagebreak
\section*{Appendix C: Succession rules and a generating tree for Bessel--Stirling numbers}
\noindent The Bessel--Stirling numbers of the second kind can also be elegantly described using the formalism of succession rules and generating trees, as introduced by West \cite{west_generating_1995}. This representation provides yet another perspective on the structure of these numbers and their recurrence relation.
\begin{definition}[Succession rule for Bessel--Stirling numbers of the second kind]
The Bessel--Stirling triangle of the second kind can be generated using the following succession rule:
\[
    \begin{aligned}
        \Omega: \begin{cases}
        (0) \\
        (2k) \to (2k)^{2k}, (2k+1) \\
        (2k+1) \to (2k+1)^{2k}, (2k+2)
        \end{cases}
    \end{aligned}
\]
where the notation $(a) \to (b)^c , (d)$ means that a node labeled $(a)$ produces $c$ children labeled $(b)$ and one child labeled $(d)$.
\end{definition}
Starting from the root node labeled $(0)$, we apply this rule iteratively to build a generating tree. The labels at level $n$ of this tree encode the Bessel--Stirling numbers $\mStirling{n}{k}_2$: the number of nodes with label $(k)$ at level $n$ corresponds precisely to $\mStirling{n}{k}_2$.
For example, the first few levels of the generating tree produce the following multisets of labels:
\[
    \begin{aligned}
    \text{Level 0}: &\ {(0)} \\
    \text{Level 1}: &\ {(1)} \\
    \text{Level 2}: &\ {(2)} \\
    \text{Level 3}: &\ {(2)^2(3)} \\
    \text{Level 4}: &\ {(2)^4(3)^4(4)} \\
    \text{Level 5}: &\ {(2)^8(3)^{12}(4)^8(5)} 
    \end{aligned}
\]
The total number of labels at level $n$ is precisely the $n$-th Bessel--Bell number (sequence A007472). This succession rule directly encodes the recurrence relation for the Bessel--Stirling numbers. For even values $k = 2j$, a node labeled $(2j)$ produces $2j$ children with the same label plus one child with label $(2j+1)$, corresponding to the term $2j \cdot \mStirling{n}{2j}_2$ in the recurrence relation. Similarly, for odd values $k = 2j+1$, a node produces $2j$ children with the same label, corresponding to the term $2j \cdot \mStirling{n}{2j+1}_2$, plus one child with label $(2j+2)$.

This succession rule formalism provides a direct connection to the combinatorial interpretation: the label $(k)$ represents using $k$ compartments, and the number of children with the same label reflects the number of valid placements in already-used compartments, while the single child with a new label represents opening a new compartment. The generalization $\Omega_m$ for arbitrary $m$ is given in Definition \ref{def-succession-general}.

\pagebreak
\begin{landscape}
\section*{Appendix D: Comparison of the classical Bell--Stirling--Touchard framework and its $m$-generalization}

\begin{table}[h]
\centering
\small
\renewcommand{\arraystretch}{1.9}
\begin{tabular}{|p{4.2cm}|p{5.2cm}|p{6.4cm}|p{6.6cm}|}
\hline
\textbf{Property} & \textbf{Classical ($m=1$)} & \textbf{Bessel case ($m=2$)} & \textbf{General $m$} \\
\hline
Defining shift property & $B_{n+1} = \sum_k \binom{n}{k}B_k$ & $B^{(2)}_{n+2} = \sum_k \binom{n}{k}2^{n-k}B^{(2)}_k$ & $B^{(m)}_{n+m} = \sum_k \binom{n}{k}m^{n-k}B^{(m)}_k$ \\
Operator equation & $\mathrm{BINOM}\, a = \mathrm{L}\, a$ & $\mathrm{BINOM}^2 a = \mathrm{L}^2 a$ & $\mathrm{BINOM}^m a = \mathrm{L}^m a$ \\
e.g.f.\ ODE & $\mathcal{A}' = e^x \mathcal{A}$ & $\mathcal{A}'' = e^{2x}\mathcal{A}$ & $\mathcal{A}^{(m)} = e^{mx}\mathcal{A}$ \\
ODE at $t=e^x$ & $tf' = tf$ & $t^2f''+tf' = t^2f$ & $(tD)^m f = t^m f$ \\
Primary carrier & $e^t$ & $I_0(t)$ & $C^{(m)}_0(t) = {}_0F_{m-1}(;1,\ldots,1;(t/m)^m)$ \\
Carrier system & $\{e^x\}$ & $\{I_0, I_1\}$ & $\{C^{(m)}_0,\ldots,C^{(m)}_{m-1}\}$ (hyper-Bessel) \\
e.g.f. & $e^{-1}e^{e^x}$ & $p\,I_0(e^x) + q\,K_0(e^x)$ & span of $m$ primitive solutions \\
Triangle recurrence & $\Stirling{n+1}{k} = k\Stirling{n}{k}+\Stirling{n}{k-1}$ & $\mStirling{n+1}{k}_2 = 2\floor{\frac{k}{2}}\mStirling{n}{k}_2+\mStirling{n}{k-1}_2$ & $\mStirling{n+1}{k}_m = m\floor{\frac{k}{m}}\mStirling{n}{k}_m+\mStirling{n}{k-1}_m$ \\
Shift identity & $\Stirling{n+1}{k} = \sum_j \binom{n}{j}\Stirling{j}{k-1}$ & $\mStirling{n+2}{k}_2 = \sum_j \binom{n}{j}2^{n-j}\mStirling{j}{k-2}_2$ & $\mStirling{n+m}{k}_m = \sum_j \binom{n}{j}m^{n-j}\mStirling{j}{k-m}_m$ \\
Row sums & $B_n$ (A000110) & A007472; parity sums A351143, A351028 & $B^{(m)}$; residue sums $= B^{(m,r)}$ \\
Dobi\'nski formula & $B_n = e^{-1}\sum_k \frac{k^n}{k!}$ & $\sum_k \frac{(2k)^n}{(2^kk!)^2} = v_nI_0(1)+u_nI_1(1)$ & $\sum_k \frac{(mk)^n}{(m^kk!)^m} = \sum_r B^{(m,r)}_n C^{(m)}_r(1)$ \\
Polynomials & $T_n(x) = \sum_k \Stirling{n}{k}x^k$ & $\mathscr{B}_n(x) = \sum_k \mStirling{n}{k}_2x^k$ & $\mathscr{B}^{(m)}_n(x) = \sum_k \mStirling{n}{k}_mx^k$ \\
Operational rep. & $e^{-x}e^{txD}e^x = e^{x(e^t-1)}$ & $G(t,x)$ of Thm.\ \ref{thm-bb-polynomials} & carrier expansion (Thm.\ \ref{thm-general-carrier-expansion}) \\
Falling factorial & $x^{\underline{n}} = \prod(x-k)$ & $x^{\uuline{n}} = \prod(x - 2\floor{k/2})$ & $x^{\uuline{n}_m} = \prod(x-m\floor{k/m})$ \\
First kind recurrence & $n\Stirlingone{n}{k}+\Stirlingone{n}{k-1}$ & $2\floor{\frac{n}{2}}\mStirlingone{n}{k}_2+\mStirlingone{n}{k-1}_2$ & $m\floor{\frac{n}{m}}\mStirlingone{n}{k}_m+\mStirlingone{n}{k-1}_m$ \\
First kind row sums & $n!$ (all permutations) & A000246 (odd-order permutations) & $\prod_{j<n}(m\floor{j/m}+1)$ \\
Lah recurrence & $(n+k)\Lah{n}{k}+\Lah{n}{k-1}$ & $2(\floor{\frac{n}{2}}{+}\floor{\frac{k}{2}})\mLah{n}{k}_2+\mLah{n}{k-1}_2$ & $m(\floor{\frac{n}{m}}{+}\floor{\frac{k}{m}})\mLah{n}{k}_m+\mLah{n}{k-1}_m$ \\
Symmetric functions & $h_{n-k}(1..k)$, $e_{n-k}(0..n{-}1)$ & $h,e$ at weights $2\floor{j/2}$ & $h,e$ at weights $m\floor{j/m}$ \\
Combinatorial model & set partitions & 2-compartment urns & $m$-compartment urns \\
Distribution & $\mathrm{Poisson}(y)$ & $\mathrm{CMP}((y/2)^2, 2)$ & $\mathrm{CMP}((y/m)^m, m)$ \\
Factorial moments & $\mathbb{E}[X^{\underline{n}}] = y^n$ & $\mathbb{E}[(2X)^{\uuline{n}}] = y^n \frac{I_{n \bmod 2}(y)}{I_0(y)}$ & $\mathbb{E}[(mX)^{\uuline{n}_m}] = y^n\frac{C^{(m)}_{n \bmod m}(y)}{C^{(m)}_0(y)}$ \\
Raw moments & $\mathbb{E}[X^n] = T_n(y)$ & $2^n\mathbb{E}[X^n] = V_n(y)/I_0(y)$ & $m^n\mathbb{E}[X^n] = V_n(y)/C^{(m)}_0(y)$ \\
\hline
\end{tabular}
\end{table}
\end{landscape}

\end{document}